\documentclass[a4paper,10pt,reqno]{amsart}
\usepackage[foot]{amsaddr}
\usepackage{amsmath}
\usepackage{amsthm, enumerate, esint}
\usepackage{mathtools}
\usepackage{hyperref,url}
\usepackage{graphicx}
\usepackage{amssymb}
\usepackage{hyperref}
\hypersetup{colorlinks=true, citecolor=blue}
\usepackage{appendix}

\usepackage[english]{babel}

\usepackage{fancyhdr}

\usepackage{calc}
\usepackage{url}

\usepackage[text={6in,8.6in},centering]{geometry}
\usepackage{bigints}

\usepackage{srcltx}

\usepackage[T1]{fontenc}

\usepackage[usenames,dvipsnames]{color}

\theoremstyle{plain}
\newtheorem{theorem}{Theorem}[section]
\theoremstyle{plain}
\newtheorem{lemma}{Lemma}[section]
\newtheorem{proposition}{Proposition}[section]

\newtheorem{definition}{Definition}[section]
\newtheorem{remark}{Remark}[section]

\renewcommand{\(}{\left(}
\renewcommand{\)}{\right)}
\renewcommand{\[}{\left[}
\renewcommand{\]}{\right]}

\DeclareMathOperator{\dist}{dist}

\newcommand{\To}{\longrightarrow}
\newcommand{\be} {\begin{equation}}
	\newcommand{\ee} {\end{equation}}
\newcommand{\bea} {\begin{eqnarray}}
	\newcommand{\eea} {\end{eqnarray}}
\newcommand{\Bea} {\begin{eqnarray*}}
	\newcommand{\Eea} {\end{eqnarray*}}

\newcommand{\al} {\alpha}
\newcommand{\ba} {\beta}
\newcommand{\de} {\delta}

\newcommand{\ga} {\gamma}
\newcommand{\Ga} {\Gamma}
\newcommand{\Om} {\Omega}

\newcommand{\De} {\Delta}
\newcommand{\la} {\lambda}

\newcommand{\no} {\nonumber}

\newcommand{\lab} {\label}
\newcommand{\va} {\varphi}
\newcommand{\var} {\varepsilon}

\newcommand{\R}{\mathbb R}

\newcommand{\Rn}{\mathbb R^N}

\newcommand{\deb}{\rightharpoonup}

\newcommand{\M}{\mathcal{M}}
\newcommand{\tM}{\tilde{\mathcal{M}}}
\newcommand{\hd}{\dot{H}^{1}(\Rn)}

\catcode`\@=11

\newcommand{\authorfootnotes}{\renewcommand\thefootnote{\@fnsymbol\c@footnote}}%
\allowdisplaybreaks

\numberwithin{equation}{section} \allowdisplaybreaks

\begin{document}
	\title{Stability of Hardy-Sobolev inequality}

\date{}

	\author[S. Chakraborty]{Souptik Chakraborty\textsuperscript{1}}
\address{\textsuperscript{1}Department of Mathematics, Indian Institute of Technology Bombay, Powai, Mumbai-400076, India}
\email{soupchak@math.iitb.ac.in}

	\keywords{Bubble, Hardy-Sobolev inequality, Linearized equation, Lyapunov-Schmidt reduction, Quantitative estimates, Stability, Struwe's decomposition.}
	
		\keywords{Bubble, Hardy-Sobolev inequality, Linearized equation, Quantitative estimates, Stability, Struwe's decomposition.}
	
	\begin{abstract}
		Given $N\geq 3,$ we consider the critical Hardy-Sobolev equation $-\De u-\frac{\ga}{|x|^2}u=\frac{|u|^{2^*(s)-2}u}{|x|^s}$ in $\Rn\setminus \{0\},$ where $0<\ga<\ga_{H}:=\left(\frac{N-2}{2}\right)^2,\,s\in (0,2)$ and $2^*(s)=\frac{2(N-s)}{(N-2)}.$
        We prove a stability estimate for the corresponding Hardy-Sobolev inequality in the spirit of Bianchi-Egnell \cite{BE}.
		\medskip
		Also, we obtain a Struwe-type decomposition \cite{S} for the corresponding Euler-Lagrange equation. Finally, we prove a quantitative bound for one bubble, namely $\dist(u,\mathcal{M})\lesssim \Ga(u)$ in the spirit of Ciraolo-Figalli-Maggi \cite{CFM}.
        %However, in the case of more than one bubble, we find %{\color{red} or ``find''? plus this sentence reads badly} 
        %quantitative estimate depends on the dimension $N$ and the parameter $s$. Precisely we show that
		%\be\no
		%\dist(u,\mathcal{M}_{k})\lesssim_{N,\ga,s} 	\Ga(u)\qquad\qquad\qquad\text{if } N<6-2s \text{  or } k=1,
		%\ee
%		{\color{red}\be\no
%			dist(u,\mathcal{M}^{\nu})\lesssim \begin{cases}
%				\Ga(u)\qquad\qquad\qquad\qquad\text{if } %\frac{\al_{+}(\ga)-\al_{-}(\ga)}{2}<2\\ 
%				N<6-2s \text{  or } \nu=1\\
%				|\Ga(u)||\log \Ga(u)|^{\tfrac{1}{2}}\qquad\,\,\,\text{if }%\frac{\ba_{+}(\ga)-\ba_{-}(\ga)}{2}=2\\
%				N=6-2s \text{  and } \nu\geq 2\\
%				|\Ga(u)|^{\tfrac{N+2-2s}{N-2}}\qquad\qquad\text{if }%\frac{\beta_{+}(\ga)-\beta_{-}(\ga)}{2}>2,
%				N>6-2s\text{  and } \nu\geq 2,
%			\end{cases}
%			\ee }
%		where $\beta_{\pm}(\ga)=\frac{N-2}{2}\pm \sqrt{\frac{(N-2)^2}{4}-\ga}$. The above findings match the results of Glaudo-Figalli \cite{FG} and Deng-Sun-Wei \cite{DSW} for $\ga=s=0$. Furthermore, we show that these estimates are sharp.
		\medskip
		
		\noindent
		{\emph{\bf 2010 MSC:} 26D10, 35A23, 35B33, 35B35, 35J20, 35J61.}
	\end{abstract}

	\maketitle
	
	\section{introduction}
	In recent years, interest has been abundant in the sharp quantitative stability of various functional and geometric inequalities and their applications to the calculus of variations, differential geometry, diffusion flow, and many other fields. Some classical examples are the Euclidean isoperimetric inequality, Brunn-Minkowski inequality, Sobolev inequality, Gagliardo-Nirenberg-Sobolev inequality, and	Caffarelli-Kohn-Nirenberg inequality to name a few from the non-exhaustive list of references \cite{BE, BL, CF, CA, CFMP, CFM, FJ, WW}. 
 
	\subsection{The Hardy-Sobolev inequality on $\Rn$}
 	In this article, our focus is to study the quantitative stability of the following Hardy-Sobolev inequality on $\Rn$ :
	\be\tag{$\mathcal{P}_{\ga,s}$}\label{HSI}
	\mu_{\ga,s}(\Rn)\left(\int_{\Rn}\frac{|u|^{2^*(s)}}{|x|^s}\,{\rm d}x\right)^{\tfrac{2}{2^*(s)}} \leq \int_{\Rn}\left(|\nabla u|^2-\ga\frac{|u|^2}{|x|^2}\right)\,{\rm d}x, \text{ for all  }u\in C^{\infty}_{c}(\Rn),
	\ee
	where $N\geq 3,\,0<\ga<\ga_{H}:=\tfrac{(N-2)^2}{4},\,0<s<2,\,2^*(s)=\tfrac{2(N-s)}{N-2}$. By density, one can extend this inequality for the Homogeneous Sobolev space $\hd$ which is the completion of $C^{\infty}_{c}(\Rn)$ with respect to $\|\nabla \cdot\|_{L^2(\Rn)}$. Here, $\mu_{\ga,s}$ is the best constant defined by
	\be\label{BCHSI}
	\mu_{\ga,s}=\mu_{\ga,s}(\Rn) := \inf_{u\in\dot{H}^1(\Rn)\setminus \{0\}}\frac{\int_{\Rn}\left(|\nabla u|^2-\ga\frac{|u|^2}{|x|^2}\right)\,{\rm d}x}{\left(\int_{\Rn}\frac{|u|^{2^*(s)}}{|x|^s}\,{\rm d}x\right)^{\tfrac{2}{2^*(s)}}}.
	\ee
    When $\ga=0,\,s=0$, \eqref{HSI} becomes the classical Sobolev inequality which is defined in \eqref{SI} and the optimal value in this case is known:
    \be\no
    \mu_{0,0}(\Rn)= S(\Rn) := \frac{N(N-2)\omega_{N}^{2/N}}{4},
    \ee
    where $\omega_N$ is the volume of the standard $N$-sphere $\mathbb{S}^N$ in $\mathbb{R}^{N+1}$. The positive minimizers of the Sobolev inequality are of the form $cU[z,\la](x)= c\la^{-\tfrac{N-2}{2}}U(\tfrac{x-z}{\la})$ for $c>0,\,\la>0$ and $z\in\Rn$ (see \cite{AT, TG}), where 
    \be\label{BSM}
    U(x)= \frac{\left(N(N-2)\right)^{\tfrac{N-2}{4}}}{(1+|x|^2)^{\tfrac{N-2}{2}}}.
    \ee
    Thus, the minimizers corresponding to Sobolev inequality form a $(N+2)$-dimensional manifold
    \be\label{MOSM}
    \tilde{\M}:= \left\{ cU[z,\la] : c\in\mathbb{R}\setminus \{0\},\,z\in\Rn,\,\la>0\right\}\subset \hd\setminus\{0\}.
    \ee
%    In this case, there is no singularity in equation \eqref{HSI}. Singularity occurs at the origin if either $s>0$ or $\ga\neq 0$.

    When $\ga=0,\, s=2$, then equation \eqref{HSI} becomes the famous Hardy inequality, where the best constant
    \be\label{HI}
    \ga_{H}(\Rn):= \mu_{0,2}(\Rn) = \frac{(N-2)^2}{4}, 
    \ee
    is never achieved on any subset of $\Rn$ including the whole Euclidean plane. If $\ga=0,\,0<s<2$ then the best constant $\mu_{0,s}(\Rn)$ and the extremal functions are known (see Catrina-Wang \cite{CW}, Ghoussoub-Yuan \cite{GY}, Lieb \cite{LAM}). More precisely,
    \be\no
    \mu_{0,s}(\Rn)= (N-2)(N-s)\left(\frac{\omega_{N-1}}{2-s}.\frac{\Ga^2(\tfrac{N-s}{2-s})}{\Ga^2(\tfrac{2N-2s}{2-s})}\right)^{\tfrac{2-s}{N-s}},
    \ee
    and a function $u\in\hd\setminus \{0\}$ is an extremal for $\mu_{0,s}(\Rn)$ if and only if $u=cV_{s}^{\la}$, for some $c\in\R\setminus\{0\}$ and $\la>0$ where,
    \be\no
    V_{s}^{\la}(x) := \left(\frac{\la^{\tfrac{2-s}{2}}}{\la^{2-s}+|x|^{2-s}}\right)^{\tfrac{N-2}{2-s}}. 
    \ee
    In this case, it is important to note that, the functions $V_{s}^{\la}$ concentrate at $0$ when $\la\to 0$.
    \medskip
    
    In full generalities, i.e., when $s\in [0,2]$ and $\ga \in (-\infty, \ga_{H}),$ after a suitable change of variables, the Hardy-Sobolev inequality is equivalent to the Caffarelli-Kohn-Nirenberg inequality (see \cite{CKN}) which says, there is a constant $C=C(a,b,N)>0$ such that
    \be\label{CKNI}
    \left(\int_{\Rn}\frac{|u|^q}{|x|^{bq}}\,{\rm d}x\right)^{\tfrac{2}{q}} \leq C \int_{\Rn} \frac{|\nabla u|^2}{|x|^{2a}}\,{\rm d}x, \quad \text{ for all }u\in C^{\infty}_{0}(\Rn),
    \ee
    where
    \be\no
    -\infty <a < \frac{N-2}{2},\quad 0\leq b-a \leq 1,\quad \text{ and   } q=\frac{2N}{N-2+2(b-a)}.
    \ee
    In full generality, that is when $0<\ga<\ga_H,\, s\in (0,2)$, the explicit value of $\mu_{\ga,s}(\Rn)$ is known (Dolbeault et al. \cite{DELT}, also see Ghoussoub-Robert \cite{GR}) and is given by
    \be\no
    \mu_{\ga,s}(\Rn)= \[(N-2)^2-4\ga\]^{\tfrac{1}{2^*(s)}+\tfrac{1}{2}}\mathcal{D}_{s},
    \ee
    where,
    \be\no
    \mathcal{D}_{s} := \left[\frac{2\pi^{\tfrac{N}{2}}}{\Ga(\tfrac{N}{2})}\right]^{\tfrac{2-s}{N-s}}\left(\tfrac{2^*(s)}{2}\right)^{\tfrac{2}{2^*(s)}}\left[\frac{\Ga\left(\tfrac{N-s}{2-s}\right)\Ga\left(\tfrac{N+2-2s}{2-s}\right)}{\Ga\left(\tfrac{2(N-s)}{2-s}\right)}\right]^{\tfrac{2-s}{N-s}}.
    \ee
    The extremizing problem \eqref{HSI} is invariant under dilation: if $U_{\ga,s}$ is an extremizer of \eqref{BCHSI}, the $U_{\ga,s}^{\la}(x):= \la^{-\tfrac{N-2}{2}}U_{\ga,s}(\tfrac{x}{\la})$ is also an extremizer of \eqref{BCHSI}, for every $\la>0$ and conversely, any extremals of $\mu_{\ga,s}(\Rn)$ are precisely a nonzero scalar multiple times $U_{\ga,s}^{\la}$, where,
	\be\label{HSBS}
	 U_{\ga,s}(x) := \frac{C_{N,\ga,s}}{\left(|x|^{\ba_{-}(\ga)\tfrac{2-s}{N-2}} + |x|^{\ba_{+}(\ga)\tfrac{2-s}{N-2}}\right)^{\tfrac{N-2}{2-s}}},\,\text{for }x\in\Rn\setminus \{0\},
	\ee
	 and $\ba_{\pm}(\ga)= \frac{N-2}{2}\pm \var$, where $\var:= \sqrt{\frac{(N-2)^2}{4}-\ga}$. Thus the extremizers of the above inequality \eqref{HSI} constitute a 2-dimensional smooth manifold
	
	\be\label{MOM}
	\mathcal{M} := \left\{c U^{\la}_{\ga,s}:\,c\in \R\setminus \{0\},\, \la>0 \right\}\subseteq \hd\setminus \{0\}.
	\ee 
	
	 The special choice of the constant $C_{N,\ga,s}:=\Big[4\left(\tfrac{N-s}{N-2}\right)\var^2\Big]^{\tfrac{1}{2^*(s)-2}}$ ensures that $U_{\ga,s}\in C^2(\Rn\setminus \{0\})\setminus \{0\}$ satisfies the Euler-Lagrange equation corresponding to \eqref{HSI}
	\be\label{MainEq}
	-\De W -\frac{\ga}{|x|^2}W =\frac{W^{2^*(s)-1}}{|x|^s}\quad \text{in }\Rn\setminus \{0\}.
	\ee
	It follows from Chou-Chu \cite{CC} that the only possible finite energy positive solutions of \eqref{MainEq} are $U_{\ga,s}^{\la}$ for any $\la>0$.
	\subsection{The Euclidean Sobolev inequality}
	The classical Sobolev inequality says for $N\geq 3$,
	\be\label{SI}
	S\left(\int_{\Rn}|u|^{2^*}\,{\rm d}x\right)^{\tfrac{2}{2^*}} \leq \int_{\Rn}|\nabla u|^2\,{\rm d}x,
	\ee
	holds for all $u\in C^{\infty}_{c}(\Rn)$ and $2^*:=\frac{2N}{N-2}$ is the critical Sobolev exponent. The inequality continues to hold for all $u$ satisfying $\|\nabla u\|_{L^2(\Rn)}<+\infty$ and $\mathcal{L}^N\big{(}\{|u|>t\}\big{)}$ for every $t>0$,  where $\mathcal{L}^N$ denotes the Lebesgue measure on $\Rn$. Here, $S\equiv S(N)>0$ is the best constant in the Sobolev inequality defined by
	\be\label{SM}
	S=S(\Rn) := \inf_{u\in\hd\setminus \{0\}} \frac{\int_{\Rn}|\nabla u|^2\,{\rm d}x}{\left(\int_{\Rn}|u|^{2^*}\,{\rm d}x\right)^{\tfrac{2}{2^*}}}.
	\ee
	
	The inequality \eqref{SI} is invariant under the conformal group action comprised of translation, dilation and inversion which causes the non-compactness in the embedding $\dot{H}^1(\Rn)\xhookrightarrow[]{} L^{2^*}(\Rn)$. More precisely, if $U$ minimizes \eqref{SI}, then any nonzero constant scalar multiple times $U[z,\la](x)= \la^{-\tfrac{N-2}{2}}U(\tfrac{x-z}{\la})$ also minimizes \eqref{SI} for any $\la>0$ and for any $z\in\Rn$. Aubin \cite{AT}, Talenti \cite{TG} computed the explicit value of $S$ and they also showed that if $u$ is a minimizer for the Sobolev constant then $u$ is a nonzero scalar multiple of $U[z,\la]$, where $U$ is given by \eqref{BSM}.
%    \be\no
%   U(x):= \left(\frac{(N(N-2))^{\tfrac{N-2}{4}}}{(1+|x|^2)^{\tfrac{N-2}{2}}}\right).
%   \ee
	\subsection{The Quantitative Stability problem and known results}
	
	Following Br\'ezis-Lieb \cite{BL} we may ask if $u\in\hd$ almost gives equality in the above Sobolev inequality \eqref{SI}, does it imply that $u$ is close to the set of all positive minimizers of \eqref{SI}? Or in other words, can we measure the discrepancy in the Sobolev inequality
	\be\no
	\de_{Eu}(u):=\|\nabla u\|_{L^2(\Rn)}^2-S\|u\|_{L^{2^*}(\Rn)}^2
	\ee
	 for $u\in\hd$, from below in terms of the $\dot{H}^1$-distance from the solution manifold? The answer is affirmative (see Bianchi-Egnell \cite{BE}, also see Rey \cite{R}). In particular, they proved that
	\begin{theorem}\textup{(\cite[Theorem~1.1]{BE})}\label{BEMT}
		There exists a positive constant $\al>0$ ($\al$ depends only on $N$),
	\be\no
	\de_{Eu}(u)^{\tfrac{1}{2}} \geq \al\, \dist (u,\tilde{\mathcal{M}}).
	\ee	
	The exponent $\frac{1}{2}$ is sharp in the sense that the inequality fails if we replace $\frac{1}{2}$ by any other exponent strictly bigger than $\frac{1}{2}$ as we consider $\de_{Eu}(u)\to 0$.
	\end{theorem} 

	Next, consider the associated Euler-Lagrange equation, which is up to a scaling given by
	\be\label{ELSI}
	\De u + u|u|^{\tfrac{4}{N-2}}=0 \qquad \text{in }\Rn.
	\ee
	In this context, the question of stability can be now reformulated as, if $u$ almost solves \eqref{ELSI}, that is whenever the quantity $\tilde{\Gamma} (u):=\|\De u + |u|^{2^*-2}u\|_{\dot{H}^{-1}(\Rn)}$ is small enough then whether $u$ is quantitatively close to Aubin-Talenti bubble (\cite{AT, TG}) given by,
	\be\no\label{ATB}
	U[z,\la](x):= \Big(N(N-2)\Big)^{\tfrac{N-2}{4}}\frac{\la^{\tfrac{N-2}{2}}}{\Big(1+\la^2|x-z|^2\Big)^{\tfrac{N-2}{2}}},\quad\text{for }\la>0,\,z\in\Rn,
	\ee 
	which are precisely all the positive solutions of \eqref{ELSI} (\cite{GNN, CGS}). Here the question is much more delicate and the answer is "No" in general, unless some extra assumptions are considered. To see this, we consider,
	$u=\sum_{i=1}^{k}U[z_i,\la_i]$, where, $\min_{\stackrel{1\leq i,j\leq k}{i\neq j}}|z_i-z_j|\gg R$ for $R>0$ very large. Then,
	\be\no
	-\De u= \sum_{i=1}^{k} -\De U[z_i,\la_i] = \sum_{i=1}^{k} U[z_i,\la_i]^{2^*-1} \approx \Big(\sum_{i=1}^k U[z_i,\la_i]\Big)^{2^*-1} = u^{2^*-1}, 
	\ee
	which gives $\tilde{\Ga} (u)\approx 0$, but clearly $u$ is not close to one single bubble $U[z_i,\la_i]$ as all the bubbles have same energy $J(U[z_i,\la_i])=\frac{1}{N} S^{\tfrac{N}{2}}$, where $J$ is the energy functional associated to \eqref{ELSI} and is defined in \eqref{EFSI}. Here it is also significant to note that the interaction between the Aubin-Talenti bubbles $U[z_i,\la_i]$ and $U[z_j,\la_j]$ for $i\neq j$ can be estimated as,
	\be\no
	\langle U[z_i,\la_i],U[z_j,\la_j]\rangle_{\dot{H}^1(\Rn)} = \int_{\Rn}U[z_i,\la_i]^{2^*-1}U[z_j,\la_j]\approx |z_i-z_j|^{-(N-2)}\ll R^{-(N-2)}=o(1)\text{ as }R\to \infty,
	\ee
	which makes $\tilde{\Ga} (u)$ small enough.
	Another delicate issue is when we take out the sign restriction on $u$. Ding \cite{DWY} proved that there are infinitely many sign-changing solutions of \eqref{ELSI} with finite energy. So, even if $\tilde{\Ga}(u)$ is small enough, $u$ may not be close to the Aubin-Talenti bubble or the sum of the Aubin-Talenti bubbles. In fact, in a seminal work due to Struwe (\cite{S}, see also \cite{SM}), he proved that if we restrict our attention to non-negative functions on $\dot{H}^1(\Rn)$ the above-mentioned bubbling phenomenon is the only other possibility due to which stability fails.
	
	\begin{theorem}\textup{(\cite{S, SM})}
		Let $N\geq 3$ and $\nu\geq 1$ be positive integers. Let $\{u_n\}_{n\geq 1}$ be a non-negative sequence in $\dot{H}^1(\Rn)$ such that $\tilde{\Ga}(u_n)\to 0$ as $n\to\infty$ and $\{u_n\}$ satisfies
		\be\no
		(\nu-\frac{1}{2})S^{\tfrac{N}{2}}\leq \|u_n\|_{\dot{H}^1(\Rn)}\leq (\nu+\frac{1}{2})S^{\tfrac{N}{2}}.
		\ee
		Then there exists a sequence of $\nu$- tuple of points $(z_1^n,\cdots,z_{\nu}^n)$ on $\mathbb{R}^{N\nu}$ and a sequence of $\nu$-tuple $(\la_1^n,\cdots,\la_{\nu}^n)\in \R_{+}^{\nu}$ of positive real numbers such that
		\be\no
		\Bigg{\|}\nabla \left(u_n-\sum_{i=1}^{\nu} U[z_i^n,\la_i^n]\right)\Bigg{\|}_{L^2(\Rn)}\to 0\qquad \text{ as }n\to\infty,
		\ee
		Moreover, the Aubin-Talenti bubbles do not interact with each other at the $\dot{H}^1(\Rn)$ level. More precisely,
		\bea
		\langle U[z_i,\la_i],U[z_j,\la_j]\rangle_{\dot{H}^1(\Rn)} &=& \int_{\Rn}U[z_i,\la_i]^{2^*-1}U[z_j,\la_j]\,{\rm d}x\no\\
		&\approx& \min_{\stackrel{1\leq i,j\leq \nu}{i\neq j}}\left(\frac{\la_i^n}{\la_j^n}, \frac{\la_j^n}{\la_i^n},\frac{1}{\la_i^n\la_j^n|z_i-z_j|^2}\right)^{\tfrac{N-2}{2}}\to 0 \text{ as }n\to \infty.\no
		\eea
		
	\end{theorem}

	The above theorem says that every non-negative Palais-Smale sequence associated with the energy functional converges to the sum of the weakly interacting family of Aubin-Talenti bubbles (see next sections for precise definitions) which gives qualitative stability result for \eqref{ELSI}. Recently, the focus has been on understanding this phenomenon more accurately and obtaining quantitative bounds on the deficit. There has been rapid progress accumulating in the following results :
	
	\begin{theorem}
		Let $N\geq 3$ and $\nu\geq 1$ be positive integers. There exists a small constant $\de\equiv \de(M,\nu)$ and a large constant $C\equiv (N,\nu)>0$ such that the following statement holds: if $u\in\dot{H}^1(\Rn)$ be such that there exist $\nu$ bubbles $\tilde{U}_1=\tilde{U}[z_1,\la_1],\cdots, \tilde{U}_{\nu}=\tilde{U}[z_{\nu},\la_{\nu}]$ $\de$-interacting Aubin-Talenti bubbles such that 
		\be\no
		\Bigg{\|}\nabla \left(u-\sum_{i=1}^{\nu}\tilde{U_{i}}\right)\Bigg{\|}_{L^2(\Rn)} \leq \de
		\ee 
		then there are $\nu$ Aubin-Talenti bubbles  $U_1=U[z_1,\la_1],\cdots, U_{\nu}=\tilde{U}[z_{\nu},\la_{\nu}]$ such that the following holds:
		\begin{itemize}
			\item \textup{(Ciraolo-Figalli-Maggi \cite{CFM})} If $\nu=1$, the,
			\be\no
			\Big{\|}\nabla \left(u-U_1\right)\Big{\|}_{L^2(\Rn)}\lesssim_{N} \Ga(u).
			\ee
			\item \textup{(Figalli-Glaudo \cite{FG})} If $\nu>1$ and $3\leq N \leq 5$, then,
			\be\label{CFMSE}
			\Big{\|}\nabla \left(u-\sum_{i=1}^{\nu}U_i\right)\Big{\|}_{L^2(\Rn)}\lesssim_{N,\nu} \Ga(u).
			\ee
			Furthermore, the interaction between the Aubin-Talenti bubbles can be estimated as
			\be\label{FGSE}
			\langle U[z_i,\la_i],U[z_j,\la_j]\rangle_{\dot{H}^1(\Rn)} = \int_{\Rn}U[z_i,\la_i]^{2^*-1}U[z_j,\la_j]\,{\rm d}x\lesssim_{N} \Ga(u).
			\ee
			Furthermore, the dimensional restriction is optimal in the sense that for $N\geq 6$ and $\nu \geq 2$, there exists a sequence $\{u_R\}\subset \dot{H}^1(\Rn)$ such that \eqref{FGSE} fails to hold for any $C \equiv C(N)$ as $R \to \infty$.
			\item \textup{(Deng-Sun-Wei \cite{DSW})} If $\nu> 1$ and $N\geq 6$, then,
			\be\label{DSWSE}
			\Big{\|}\nabla \left(u-\sum_{i=1}^{\nu}U_i\right)\Big{\|}_{L^2(\Rn)}\lesssim_{N,\nu}
			\begin{cases} \Ga(u)\log |\Ga(u)|^{\tfrac{1}{2}}\qquad\text{if }N=6\\
				\Ga(u)^{\tfrac{N+2}{2(N-2)}}\qquad\quad\,\text{ if }N\geq 7.
			\end{cases}
			\ee
			
		\end{itemize}
	Furthermore, all the above estimates are sharp.
	\end{theorem}
	
	\subsection{Quantitative Stability of Hardy-Sobolev}
	In this article, our focus is to study the quantitative stability of the Hardy-Sobolev inequality \eqref{HSI} and the stability associated to its Euler-Lagrange equation on $\Rn$. First, we prove the stability of \eqref{HSI} around the minimizers (which is precisely the manifold $\M$ defined in \eqref{MOM}) in the spirit of Bianchi-Egnell \cite{BE} (also see \cite{RSW}, for similar results with $\ga=0,\,s\in (0,2)$). Then we prove stability around a critical point in the spirit of Ciraolo-Figalli-Maggi \cite{CFM}. First, we prove a qualitative stability result, namely Theorem \ref{PDHSE}.
%	\be\tag{$\mathcal{P}_{\ga,s}$}\no
%	\mu_{\ga,s}(\Rn)\left(\int_{\Rn}\frac{|u|^{2^*(s)}}{|x|^s}\,{\rm d}x\right)^{\tfrac{2}{2^*(s)}} \leq \int_{\Rn}\left(|\nabla u|^2-\ga\tfrac{|u|^2}{|x|^2}\right)\,{\rm d}x; \text{  for all  }u\in C^{\infty}_{c}(\Rn),
%	\ee
%	where $N\geq 3,\,0<\ga<\ga_{H}:=\tfrac{(N-2)^2}{4},\,0<s<2,\,2^*(s)=\tfrac{2(N-s)}{N-2}$. The extremizers of the above inequality constitute a 2-dimensional smooth manifold
	
%	\be\no
%	\mathcal{M} := \left\{c U_{\ga,s}^{\la}:\,c\in \R\setminus \{0\},\, \la>0 \right\}\subseteq \dot{H}^1(\Rn)\setminus \{0\}.
%	\ee  	
	Here,
	\be\label{HSBU}
	U_{\ga,s}(x) := \frac{C_{N,\ga,s}}{\left(|x|^{\ba_{-}(\ga)\tfrac{2-s}{N-2}} + |x|^{\ba_{+}(\ga)\tfrac{2-s}{N-2}}\right)^{\tfrac{N-2}{2-s}}}
	\ee
	and $U_{\ga,s}^{\la}(x):= \la^{-\tfrac{N-2}{2}}U_{\ga,s}(\tfrac{x}{\la})$ for $\la>0$ and $\ba_{\pm}(\ga):= \frac{N-2}{2}\pm \var$ with $\var=\sqrt{\frac{(N-2)^2}{4}-\ga}$.
	We obtain the non-quantitative stability result for the Hardy-Sobolev inequality as we establish the Struwe-type profile decomposition theorem for the Palais-Smale sequences (Theorem \ref{PDHSE}). 
 %Using this we prove the stability results associated with the Euler-Lagrange equation of \eqref{HSI} in the spirit of Ciraolo-Figalli Maggi \cite{CFM}, Figalli-Glaudo \cite{FG} and Deng-Sun-Wei \cite{DSW} (also see \cite{WW}). 

    \begin{remark}\label{eqnorm}
{\rm	For $0<\ga<\ga_{H},$
	\be
\|u\|_{\ga}:=\left(\int_{\Rn}|\nabla u|^2\,{\rm d}x-\ga\int_{\Rn}\frac{|u|^2}{|x|^{2}}\;{\rm d}x\right)^{\tfrac{1}{2}}\no
	\ee
	defines a norm in $\hd$ which is equivalent to the standard norm in $\hd$. In particular,
	$$\sqrt{1-\frac{\ga}{\ga_{H}}}\|u\|_{\hd}\leq \|u\|_{\ga}\leq\|u\|_{\hd}.$$
The corresponding equivalent inner product $\langle \cdot,\cdot\rangle_{\ga}$ in the homogeneous Hilbert space $\hd$ is given by
	\be
	\langle u,v\rangle_{\ga}:=\int_{\Rn}\nabla u\cdot\nabla v\,{\rm d}x-\ga\int_{\Rn}\frac{uv}{|x|^{2}}\;{\rm d}x.\no
	\ee
Finally, for simplicity, we endow in what follows the weighted Lebesgue space $L^{2^*(s)}(\Rn, |x|^{-s})$
with the norm $\|u\|_{L^{2^*(s)}(\Rn, |x|^{-s})}=\left(\int_{\Rn}\frac{|u|^{2^*(s)}}{|x|^s}\;{\rm d}x\right)^{\tfrac{1}{2^*(s)}}$.}
\end{remark}
	
	\subsection{Main results}
    
	\begin{theorem}\label{BEHSE}
		There exists a positive constant $\al\equiv \al(N,\ga,s)$ such that
		\be\no
		\|u\|_{\ga}^2-\mu_{\ga,s}(\Rn)\|u\|_{L^{2^*(s)}(\Rn,|x|^{-s})}^2 \geq \al
		\, \textup{dist}_{\dot{H}^1(\Rn)}(u,\mathcal{M})^2,\,\text{for all }u\in\dot{H}^1(\Rn).
		\ee
		The exponent is sharp in the sense stated in Theorem \ref{BEMT}.
	\end{theorem}

 %   \begin{theorem}\label{MBMR}
 %      Let $\{u_{\al}\}_{\al}$ be nonnegative sequence in $\hd$ such that 
 %       \be\no
 %       \begin{cases}
 %           \Gamma (u_{\al}):=\|\De u_{\al}+\frac{\ga}{|x|^2}u_{\al}+\frac{u_{\al}^{2^*(s)-1}}{|x|^s}\|_{H^{-1}(\Rn)} = o(1)\\
 %           \Big(\nu-\frac{1}{2}\Big)\mu_{\ga,s}(\Rn)^{\tfrac{N-s}{2-s}}\leq\int_{\Rn}\left\{|\nabla u_{\al}|^2-\frac{\ga}{|x|^2}|u_{\al}|^2\right\}\,{\rm d}x \leq \Big(\nu+\frac{1}{2}\Big)\mu_{\ga,s}(\Rn)^{\tfrac{N-s}{2-s}}
 %       \end{cases}
 %       \ee
 %      Then, there exists sequence of $\nu$-tuples $\left\{(\la_{\al,1},\cdots,\la_{\al,\nu})\right\}\subset \R^{\nu}_{+}$ such that
 %       \be\no
 %       \|u_{\al}-\sum_{i=1}^{\nu}U_{\ga,s}^{\la_{\al,i}}\|_{\ga}^2 \lesssim_{N,\ga,s}
 %           O\big(\Ga(u_{\al})\big)^2\quad \text{ for }\nu=1,\text{ or }3\leq N < 6-2s.\\
 %       \ee
 %   \end{theorem}
    In the spirit of Ciralo-Figalli-Maggi \cite{CFM} we prove the following :
    \begin{theorem}\label{CFMHS}
        Let $u\in\hd$ be a nonnegative function satisfying
        \be\no
        \frac{1}{2}\mu_{\ga,s}^{\tfrac{N-s}{2-s}}\leq \|u\|_{\ga}\leq \frac{3}{2}\mu_{\ga,s}^{\tfrac{N-s}{2-s}}
        \ee
        Then there exists $\la>0$ and $\rho\in\hd$ such that
        \be\no
        u=U_{\ga,s}^{\la}+\rho,
        \ee
        with $\|\rho\|_{\ga}\lesssim \Ga(u)$.
    \end{theorem}
%    The above result matches the main result of Ciraolo-Figalli-Maggi (see \cite{CFM}, Theorem $1.1$) for $\ga=s=0$.
	%\subsection{Main difficulty, novelty and strategy of the proof}
	In a forthcoming article, we deal with the case of multibubbles. Also, we remark that sharp stability estimates for CKN inequality and its associated Euler-Lagrange equation have been done recently by Wei-Wu \cite{WW}.

    \subsection{Structure of the paper} 
    In section \ref{s2}, we prove the Bianchi-Egnell stability result \cite{BE} around a positive minimizer corresponding to the Hardy-Sobolev inequality, namely Theorem \ref{BEHSE}. In section \ref{SPDHS}, we prove the profile decomposition result for the Palais-Smale sequences associated with the energy functional to the Hardy-Sobolev equation. In section \ref{SCFMHS}, we prove the quantitative stability result for the Hardy-Sobolev equation when there is one bubble under consideration, namely Theorem \ref{CFMHS}. In section \ref{SAPP}, we show the spectral properties of the corresponding linearized operator around a bubble which plays a key role in proving Theorem \ref{BEHSE} and Theorem \ref{CFMHS} and also computed integral estimation for interaction between bubbles. 

    \medskip

    \medskip
    
	\section{Bianchi-Egnell's result for the H-S inequality}\label{s2}
		The following lemma plays a pivotal role in the proof of Theorem \ref{BEHSE}. This lemma gives a precise picture of how the functions in $\hd$ which are near to $\M$ (defined in \eqref{MOM}) behave. Here $U_{\ga,s}$ is a minimizer of \eqref{BCHSI} and $C_{N,\ga,s}$ is chosen so that $\|U_{\ga,s}\|_{\ga}=1$.

	\begin{lemma}\label{L2.1}
		There is a constant $\al=\al(N,\ga,s)>0$ such that
		\be\label{BEHS1}
		\|\phi\|_{\ga}^2-\mu_{\ga,s}(\Rn)\|\phi\|^2_{L^{2^*(s)}(\Rn,|x|^{-s})}\geq \al \dist (\phi,\mathcal{M})^2 +o\left(\dist (\phi,\mathcal{M})^2\right),
		\ee
		for all $\phi \in\dot{H}^1(\Rn)$ with $\dist (\phi,\mathcal{M})<\|\phi\|_{\ga}$.
	\end{lemma}

	\begin{proof}
		$\mathcal{M}$ is a two-dimensional manifold embedded inside $\dot{H}^1(\Rn)$:
		\bea
		\mathcal{M} &\hookrightarrow& \dot{H}^1(\Rn)\no\\
		\R\setminus \{0\}\times\R_{+}\times \M\ni (c,\la,U_{\ga,s})&\mapsto& cU_{\ga,s}^{\la}\in \dot{H}^{1}(\Rn).\no
		\eea
		Take $\phi\in\dot{H}^1(\Rn)$ with 
		\bea
		\dist(\phi,\mathcal{M})^2 &=&\inf_{(c,\la)\in \R\setminus\{0\}\times \R_+} \|\phi-cU_{\ga,s}^{\la}\|_{\ga}^2\no\\
		&=&\inf_{(c,\la)\in \R\setminus\{0\}\times \R_+}\left[\|\phi\|_{\ga}^2 + |c|^2\|U_{\ga,s}^{\la}\|_{\ga}^2-2\langle \phi,cU_{\ga,s}^{\la}\rangle_{\ga}\right] < \|\phi\|_{\ga}^2\label{DFA}\\
		&&\qquad\qquad\qquad\text{(by hypothesis of lemma)}\no
		\eea
		Let $\{(c_k, \la_k)\}_{k\in\mathbb{N}}$ be a minimizing sequence. Define,
		$$v_k:= \|\phi\|_{\ga}^2+|c_k|^2\|U_{\ga,s}^{\la_k}\|_{\ga}^2-2\langle \phi, c_kU_{\ga,s}^{\la_k}\rangle_{\ga}\to \dist(\phi,\M)^2\,\,\text{ as }k\to\infty.$$
		
		Again from \eqref{DFA}, it follows that
		\be\label{DFA1}
		\limsup_{k\to\infty}v_k <\|\phi\|_{\ga}^2
		\ee
		$v_k$ being a quadratic polynomial in $c_k$, hence $c_k$ is bounded and thus converges (up to a subsequence, still denoting by $c_k$) to some $c_0$ as $k\to\infty$. Again $c_0>0$, follows from \eqref{DFA1}.
		Now we have,
		\bea
		\langle \phi,U_{\ga,s}^{\la_k}\rangle_{\ga}&=&\int_{\Rn} \nabla \phi (x)\cdot \nabla U_{\ga,s}^{\la_k}(x)\,{\rm d}x - \ga\int_{\Rn}\frac{\phi(x)U_{\ga,s}^{\la_k}(x)}{|x|^2}\,{\rm d}x\no\\
		&=&\int_{\Rn} \la_k^{\tfrac{N}{2}}\nabla \phi (x)\cdot \nabla U_{\ga,s}(\la_k x)\,{\rm d}x - \ga\int_{\Rn} \frac{\la_k^{\tfrac{N-2}{2}}\phi(x)U_{\ga,s}(\la_k x)}{|x|^2}\,{\rm d}x\no\\
		&=&\int_{\Rn} \la_k^{-\tfrac{N}{2}}\nabla \phi (\tfrac{x}{\la_k})\cdot \nabla U_{\ga,s}(x)\,{\rm d}x - \ga\int_{\Rn} \frac{\la_k^{-\left(\tfrac{N-2}{2}\right)}\phi(\tfrac{x}{\la_k})U_{\ga,s}(x)}{|x|^2}\,{\rm d}x.\no
		\eea
		Again from the scale invariance of $\|\cdot\|_{\ga}$ under the conformal transformation we have for all $k$, $\|U_{\ga,s}^{\la_k}\|_{\ga}=\|U_{\ga,s}\|_{\ga}$. Thus if $\la_k\to\infty$ as $k\to\infty,$ then $U_{\ga,s}^{\la_k}\rightharpoonup 0$ in $\hd$ as $k\to\infty$. Then $v_{k}\to \|\phi\|_{\ga}^2+c_0^2\|U_{\ga,s}\|_{\ga}^2$ which contradicts \eqref{DFA1}. Hence, $\{\la_k\}_{k\in\mathbb{N}}$ must be bounded and thus $\la_k\to \la_0\in \R_{+}\cup\{0\}$ (up to a subsequence). Now, $\la_0=0$ contradicts \eqref{DFA1}. Thus, $\la_0>0.$ 
		
		Since $\mathcal{M}$ is a smooth manifold, we must get $(\phi-c_0U_{\ga,s}^{\la_0})\perp T_{c_0U_{\ga,s}^{\la_0}}\mathcal{M}$. Furthermore, $T_{c_0U_{\ga,s}^{\la_0}}\mathcal{M} = \operatorname{span} \left\{U_{\ga,s}^{\la_0},\,\frac{{\rm d}}{{\rm d}\la}\bigg{|}_{\la=\la_0}U_{\ga,s}^{\la}\right\}.$ Now consider the operator, 
		\be\label{Lgar}
		\mathcal{L}_{\ga,\la_0} := \frac{\left(-\De-\frac{\ga}{|x|^2}\right)}{\frac{(U_{\ga,s}^{\la_0})^{2^*(s)-2}}{|x|^s}}\text{ on }L^2\left(\frac{(U_{\ga,s}^{\la_0})^{2^*(s)-2}}{|x|^s},\Rn\right).
		\ee
		 Since the weight function, $\frac{(U_{\ga,s}^{\la})^{2^*(s)-2}}{|x|^s}\in L^{\tfrac{N}{2}}(\Rn)$ for any $\la>0$, following similar argument as in \cite{FG} (Appendix $A$, Proposition $A.2$), $\left(\frac{(-\De-\frac{\ga}{|x|^2})}{\frac{(U_{\ga,s}^{\la})^{2^*(s)-2}}{|x|^s}}\right)^{-1}$ is well-defined and continuous from $L^2\left(\frac{(U_{\ga,s}^{\la})^{2^*(s)-2}}{|x|^s},\Rn\right)$ into $\dot{H}^1(\Rn)$. Furthermore, it is a compact, self-adjoint operator from $L^2\left(\frac{(U_{\ga,s}^{\la})^{2^*(s)-2}}{|x|^s},\Rn\right)$ into itself. Hence, its spectrum is discrete. Thus we have the following Rayleigh quotient characterization for eigenvalues 
		\be\no
		\eta_{m}= R(w_m):= \frac{\|w_m\|_{\ga}^2}{\int_{\Rn}\tfrac{(U_{\ga,s}^{\la_0})^{2^*(s)-2}}{|x|^s}|w_m|^2\,{\rm d}x},
		\ee
		where $w_m$ denotes the $m$-th eigenfunction of $\mathcal{L}_{\ga,\la_0}$ and further if we consider $V_m= \operatorname{span} \{w_1,\cdots,w_m\},$ we have the following variational characterization for the eigenvalues
		\bea
		\eta_m =\min_{\stackrel{v\in V_m}{v\neq 0}}R(v) &=& \min_{v\perp V_{m-1}}R(v)\no\\
		&=&\min_{\stackrel{W\subseteq \dot{H}^1(\Rn)}{\dim (W)=m}} \max_{\stackrel{v\in W}{v\neq 0}}R(v).\no
		\eea
		In particular, $$\eta_1=\min_{\stackrel{v\in\dot{H}^1(\Rn)}{v\neq 0}}R(v).$$
		 From the above characterization of eigenvalues, we have,
		\be\label{3ev}
		\eta_3 \leq \frac{\|w\|_{\ga}^2}{\int_{\Rn}\frac{(U_{\ga,s}^{\la_0})^{2^*(s)-2}}{|x|^s}w^2\,{\rm d}x}\qquad\text{for all }w\perp T_{c_0U_{\ga,s}^{\la_0}}\mathcal{M},
		\ee
		with equality if and only if $w$ is the third eigenfunction. We also have $\eta_1=\mu_{\ga,s}(\Rn)^{\tfrac{2^*(s)}{2}},\,\eta_2=(2^*(s)-1)\mu_{\ga,s}(\Rn)^{\tfrac{2^*(s)}{2}}$ and all eigenvalues are independent of the scaling factor $\la>0$. Moreover,
		$$T_{U_{\ga,s}^{\la_0}}\mathcal{M}= \operatorname{span} \{U_{\ga,s}^{\la_0}\} \bigoplus \operatorname{span} \left\{\frac{{\rm d}}{{\rm d}\la}\bigg{|}_{\la=\la_0}U_{\ga,s}^{\la}\right\},$$
		that is the direct sum of the first and second eigenspace of the operator $\mathcal{L}_{\ga,\la_0}$. Since $(\phi-c_0U_{\ga,s}^{\la_0})\perp T_{c_0U_{\ga,s}^{\la_0}}\mathcal{M},$ we can write, 
		$$\phi = c_0U_{\ga.s}^{\la_0}+dv,$$
		where $v\perp T_{U_{\ga,s}^{\la_0}}\mathcal{M}$ with $\|v\|_{\ga}=1$ and $d=\dist(\phi,\mathcal{M})$. Now, using Taylor series expansion around $U_{\ga,s}^{\la_0}$ yields,
		\bea
		\int_{\Rn}\frac{|\phi|^{2^*(s)}}{|x|^s}\,{\rm d}x &=& \int_{\Rn}\frac{|c_0U_{\ga,s}^{\la_0}+dv|^{2^*(s)}}{|x|^s}\,{\rm d}x\no\\
		&=&|c_0|^{2^*(s)}\int_{\Rn}\frac{|U_{\ga,s}^{\la_0}|^{2^*(s)}}{|x|^s}\,{\rm d}x + 2^*(s) d|c_0|^{2^*(s)-1}\int_{\Rn}\frac{|U_{\ga,s}^{\la_0}|^{2^*(s)-2}}{|x|^s}U_{\ga,s}^{\la_0}v\,{\rm d}x\no\\
		&\,&\quad + \frac{2^*(s)(2^*(s)-1)|c_0|^{2^*(s)-2}d^2}{2!}\int_{\Rn} \frac{(U_{\ga,s}^{\la_0})^{2^*(s)-2}}{|x|^s}v^2\,{\rm d}x + o(d^2)\no\\
		&\leq& |c_0|^{2^*(s)}\mu_{\ga,s}(\Rn)^{-\tfrac{2^*(s)}{2}}+ |c_0|^{2^*(s)-2} \frac{2^*(s)(2^*(s)-1)d^2}{2!}\frac{1}{\eta_3}+o(d^2).\no
		\eea
		Since, $U_{\ga,s}^{\la_0}$ solves
			\be\label{MEQWHSC}
		-\De W -\frac{\ga}{|x|^2}W =\mu_{\ga,s}(\Rn)^{\tfrac{2^*(s)}{2}}\frac{W^{2^*(s)-1}}{|x|^s}\quad \text{in }\Rn\setminus \{0\}.
		\ee
		 and $v\perp U_{\ga,s}^{\la_0}$, 
		$$0=\langle U_{\ga,s}^{\la_0},v\rangle_{\ga}=\int_{\Rn}\frac{(U_{\ga,s}^{\la_0})^{2^*(s)-2}}{|x|^s}U_{\ga,s}^{\la_0} v\,{\rm d}x.$$
		Now, as $\eta_2= (2^*(s)-1)\mu_{\ga,s}(\Rn)^{\tfrac{2^*(s)}{2}}$, we obtain
		\be\no
		\int_{\Rn}\frac{|\phi|^{2^*(s)}}{|x|^s}\,{\rm d}x\leq |c_0|^{2^*(s)}\mu_{\ga,s}(\Rn)^{-\tfrac{2^*(s)}{2}}+ \frac{2^*(s)d^2}{2!}|c_0|^{2^*(s)-2}\mu_{\ga,s}(\Rn)^{-\tfrac{2^*(s)}{2}}\frac{\eta_2}{\eta_3}+o(d^2).
		\ee
		Hence, 
		\bea
		\|\phi\|_{L^{2^*(s)}(\Rn,|x|^{-s})}^2 &\leq& |c_0|^2\mu_{\ga,s}(\Rn)^{-1}\left(1+\frac{2^*(s)d^2}{2}\frac{\eta_2}{\eta_3}+o(d^2)\right)^{\tfrac{2}{2^*(s)}}\no\\
		&\leq& |c_0|^2\mu_{\ga,s}(\Rn)^{-1}\left(1+d^2|c_0|^{-2}\frac{\eta_2}{\eta_3}+o(d^2)\right),\no
		\eea
		where in the last inequality we have used $(1+t)^{p}\leq 1+ pt,$ holds for $t>0$ and $p\in (0,1)$.
		Hence we obtain the desired expression
		\be\no
		\|\phi\|_{\ga}^2-\mu_{\ga,s}(\Rn)\|\phi\|_{L^{2^*(s)}(\Rn,|x|^{-s})}^2\geq d^2 \left(1-\frac{\eta_2}{\eta_3}\right) + o(d^2).
		\ee
		Thus the lemma holds for $\al = \left(1-\frac{\eta_2}{\eta_3}\right)>0$. To see this estimate is sharp we argue as follows.
		
		Take $\phi= U_{\ga,s}+dv$ where $v$ is the third eigenfunction of $\mathcal{L}_{\ga,1}$ and $d$ is a small positive number. Then if $d$ is small enough we get $\dist(\phi,\M)=d$ and the closest point on $\M$ is $U_{\ga,s}$. Thus
		\be\no
		\eta_3= \frac{\|\phi\|_{\ga}^2}{\int_{\Rn}\frac{U_{\ga,s}^{2^*(s)-2}}{|x|^s}v^2\,{\rm d}x}=\frac{1+d^2}{\int_{\Rn}\frac{U_{\ga,s}^{2^*(s)-2}}{|x|^s}v^2\,{\rm d}x}.
		\ee
		It is clear that the last claim holds locally. If we assume that the claim is not true globally for any $d>0$, then there exists a sequence $(c_m,\la_m)\nrightarrow (1,1)$ such that $c_mU_{\ga,s}^{\la_m}\to U_{\ga,s}$ in $\dot{H}^1(\Rn),$ which is absurd.\\
		Now the same argument as above yields
		$$\|\phi\|_{\ga}^2-\mu_{\ga,s}(\Rn)\|\phi\|_{L^{2^*(s)}(\Rn,|x|^{-s})}^2= d^2\left(1-\frac{\eta_2}{\eta_3}\right)+o(d^2).$$
	\end{proof}
	\begin{proof}[Proof of Theorem \eqref{BEHSE}]
		The sharpness of the theorem follows from the proof of the last part of the previous lemma.
		
		We have the local result from the previous lemma in a region around $\mathcal{M}$. Now we are going to use a Concentration Compactness argument to show that some stability estimate must hold outside this local region also. Since we are working on $\Rn$, we can cite this step as a straightforward application of Concentration Compactness as presented by P.L. Lions \cite{PLL} or M. Struwe \cite{S}. 
		
		Assume that the Theorem does not hold. Then we would be able to find a sequence $\{u_m\}_{m\in\mathbb{N}}$ such that
			$$\frac{\|u_m\|_{\ga}^2-\mu_{\ga,s}(\Rn)\|u_m\|_{L^{2^*(s)}(\Rn,|x|^{-s})}^2}{\dist(u_m,\mathcal{M})^2}\to 0\,\,\text{as }m\to\infty.$$
		By homogeneity, we can assume $\|u_m\|_{\ga}=1$ and after selecting a subsequence we can also assume that $\dist(u_m,\mathcal{M})\to \theta\in \[0,1\]$. We note that, $\dist(u_m,\mathcal{M})\leq \|u_m\|_{\ga}=1.$
		
		If $\theta=0,$ then we get using Lemma \eqref{L2.1},
			$$\frac{\|u_m\|_{\ga}^2-\mu_{\ga,s}(\Rn)\|u_m\|_{L^{2^*(s)}(\Rn,|x|^{-s})}^2}{\dist(u_m,\mathcal{M})^2}\geq \al + o(1)>0,$$
			which contradicts our assumption.
		
		The other possibility is $L>0.$ In this case, we must have, $\|u_m\|_{\ga}^2-\mu_{\ga,s}(\Rn)\|u_m\|_{L^{2^*(s)}(\Rn,|x|^{-s})}^2\to 0$ as $m\to\infty,\,\|u_m\|_{\ga}=1$ for each $m\in\mathbb{N}$. Now scaling argument as in Proposition \ref{PDHSE} (see also Struwe \cite{S}) yields a sequence $\{\la_m\}_{m\in\mathbb{N}}$ so that $\la_m^{\tfrac{N-2}{2}}U_{\ga,s}(\la_m x)\to \pm U_{\ga,s}$ in $\dot{H}^1(\Rn)$ as $m\to\infty.$ This implies that $\dist(u_m,\M)=\dist(\la_m^{\tfrac{N-2}{2}}u(\la_m x),\M)\to 0$ as $m\to\infty.$ This gives the desired contradiction.
		
	\end{proof}

	\medspace

	\section{Profile decomposition for Hardy-Sobolev equation}\label{SPDHS}
	
		In the seminal work due to Struwe \cite{S}, he proved that all critical points at infinity of the functional $J$ are induced by limits of sums of Talenti bubbles, where 
		\be\label{EFSI}
		J(u):= \frac{1}{2}\int_{\Rn}|\nabla u|^2\,{\rm d}x-\frac{1}{2^*}\int_{\Rn}|u|^{2^*}\,{\rm d}x,\qquad u\in\dot{H}^1(\Rn),
		\ee
	whose Euler-Lagrange equation is $\De u+|u|^{2^*-2}u=0$ in $\Rn.$
	Here we will give an analog of the above result for the critical Hardy-Sobolev inequality.
	
	In this section, we study the Palais-Smale sequences of the problem \eqref{MainEq}, where $0<\ga<\ga_{H},\,s\in (0,2)$. Define the associated energy functional 
	\bea\label{EF}
	I_{\ga,s}(u) &:=& \frac{1}{2}\int_{\Rn}\[|\nabla u|^2-\ga\frac{|u|^2}{|x|^2}\]\,{\rm d}x - \frac{1}{2^*(s)}\int_{\Rn}\frac{|u|^{2^*(s)}}{|x|^s}\,{\rm d}x\no\\
	&=& \frac{1}{2}\|u\|_{\ga}^2 - \frac{1}{2^*(s)}\int_{\Rn}\frac{|u|^{2^*(s)}}{|x|^s}\,{\rm d}x.
	\eea
	
	We say that the sequence $(u_n)_n\in \dot{H}^1(\Rn)$ is a $(PS)$ sequence for $I_{\ga,s}$ at level $\ba$ if $I_{\ga,s}(u_n)\to \ba$ and $(I_{\ga,s})'(u_n)\to 0$ in $H^{-1}(\Rn)$. It is easy to see that the weak limit of a $(PS)$ sequence solves \eqref{MainEq} except for the positivity.
	
	However, the main difficulty is that the $(PS)$ sequence may not converge strongly and hence the weak limit can be zero even if $\ba>0.$
	The main purpose of this section is to classify $(PS)$ sequences of
	the functional $I_{\ga,s}$. Classification of $(PS)$
	sequences have been done for various problems having a lack of compactness, to quote a few, we cite in the local case \cite{BS, Sm}  with Hardy potentials and in~\cite{S} without Hardy potentials.
	
	\begin{theorem}\label{PDHSE}
		Let $\{u_n\}$ be a Palais-Smale sequence for $I_{\ga,s}$ at level $\ba$. Then up to a subsequence \textup{(still denoted by $u_n$)} the following properties hold:
		there exists $m\in\mathbb{N}_0=\mathbb{N}\cup\{0\}$, $m$ sequences $\{R_n^k\}_n\subset \mathbb{R}_{+}$ \textup{(for $1\leq k\leq m$)} and $u_0\in\dot{H}^1(\Rn)$ such that
		\bea
		&(i)&\, u_n = u_0+\sum_{k=1}^{m}\(U_{\ga,s}\)^{R_n^k}+o(1),\no\\
		&(ii)&\, I'_{\ga,s}(u_0)=0\no\\
		&(iii)&\, R_n^k\to 0\,\text{ or, }R_n^k\to\infty,\,\frac{R_{n}^{k+1}}{R_{n}^{k}}\to\infty\text{ as }k\to \infty\no\\
		&(iv)&\, \ba = I_{\ga,s}(u_0)+\sum_{k=1}^{m} I_{\ga,s}(U_{\gamma,s}^{R_n^k})+o(1),\no
		\eea
		where $o(1)\to 0$ in $\dot{H}^1(\Rn)$.
	\end{theorem}
	
	\begin{proof}
		We prove the theorem in several steps.
		
		\noindent\underline{\bf Step 1:} Using standard arguments it follows that there exists $M>0$ such that $$\sup_{n\in\mathbb{N}}\|u_n\|_{\ga}\leq M.$$
		More precisely, as $n\to\infty$
		\bea
		\ba + o(1)+ o(1)\|u_n\|_{\gamma} &\geq& I_{\ga,s}(u_n) \, - \,
		\frac{1}{2^*(s)} \prescript{}{H^{-1}}{\big\langle}I'_{\ga,s}(u_n), u_n{\big\rangle}_{\dot{H}^1}\no\\
		&=& \left(\frac{1}{2}-\frac{1}{2^*(s)}\right)\|u_n\|_{\gamma}^{2}.\no
		\eea
		As $2^*(s)>2$, from the above estimate, it follows that $(u_n)_n$ is bounded in $\dot{H}^1(\Rn)$. Consequently, there exists $u_0$ in $\dot{H}^1(\Rn)$
		such that, up to a subsequence, still denoted by $(u_n)_n$, $u_n\rightharpoonup u_0$ in $\dot{H}^1(\Rn)$ and $u_n\to u_0$ a.e. in $\Rn$.  Moreover, as  $\prescript{}{H^{-1}}{\big\langle}I'_{\ga,s}(u_n), v{\big\rangle}_{\dot H^1}\rightarrow 0$ as $n\rightarrow\infty$ for all  $v\in\dot{H}^1(\Rn)$, then
		\be\label{B6}
		-\De u_n -\gamma \frac{u_n}{|x|^{2}} - \frac{|u_n|^{2^*(s)-2}u_n}{|x|^s}\To 0\quad \mbox{in}\quad H^{-1}(\Rn).
		\ee
		\underline{\bf Step 2:}
		From \eqref{B6}, letting $n \rightarrow \infty$, we get
		\begin{equation}\lab{24-5-1}
			\langle u_n, v\rangle_{\gamma}-\int_{\Rn}\frac{|u_n|^{2^*(s)-2}u_nv}{|x|^s}\;{\rm d}x \, {\rightarrow} \,  0,\qquad\text{for all }v\in\dot{H}^1(\Rn).
		\end{equation}
		As $u_n\rightharpoonup u_0$ in $\dot{H}^1(\Rn)$, it follows that $\langle u_n, v\rangle_{\gamma}\to \langle u_0, v\rangle_{\gamma}$ for all $v\in\dot{H}^1(\Rn)$.
		
		\noindent
		{\bf Claim 1}: $\displaystyle\int_{\Rn}\frac{|u_n|^{2^*(s)-2}u_nv}{|x|^s}\;{\rm d}x \To \int_{\Rn}\frac{|u_0|^{2^*(s)-2} u_0 v}{|x|^s} {\rm d}x$
		for all $v\in\dot{H}^1(\Rn)$.
		
		Indeed, $u_n\to u_0$ a.e. in $\Rn$ and
		\be\lab{24-5-2}\int_{\Rn}\frac{|u_n|^{2^*(s)-2}u_nv}{|x|^s}\;{\rm d}x = \int_{B_R}\frac{|u_n|^{2^*(s)-2}u_nv}{|x|^s}\;{\rm d}x +\int_{\Rn\setminus B_R}\frac{|u_n|^{2^*(s)-2}u_nv}{|x|^s}\;{\rm d}x. \ee
		On $B_R$ we are going to use Vitali's convergence theorem. For that, given any $\varepsilon>0$, we choose $\Om\subset
		B_R$ such that
		$\displaystyle
		\bigg(\int_{\Omega}\frac{|v|^{2^*(s)}}{|x|^s}{\rm d}x\bigg)^{\tfrac{1}{2^*(s)}}<\frac{\varepsilon}{(M\mu_{\ga,s}^{-\frac{1}{2}})^{2^*(s)-1}}$. Since $\frac{|v|^{2^*(s)}}{|x|^s}$ is in  $L^1(\Rn)$, the above choice makes sense. Therefore using H\"older's inequality and Hardy-Sobolev's inequality we get,
		
		\bea
		\bigg{|}\int_{\Omega}\frac{|u_n|^{2^*(s)-2}u_nv}{|x|^s}\;{\rm d}x\bigg{|} &\leq& \int_{\Omega}\frac{|u_n|^{2^*(s)-1}|v|}{|x|^s}\;{\rm d}x \no\\
		&\leq& \bigg(\int_{\Omega}\frac{|u_n|^{2^*(s)}}{|x|^s}{\rm d}x\bigg)^{\tfrac{2^*(s)-1}{2^*(s)}}\bigg(\int_{\Omega}\frac{|v|^{2^*(s)}}{|x|^s}{\rm d}x\bigg)^{\tfrac{1}{2^*(s)}}\no\\
		&\leq& \mu_{\ga,s}^{-\frac{2^*(s)-1}{2}}\|u_n\|_{\gamma}^{2^*(s)-1}\bigg(\int_{\Omega}\frac{|v|^{2^*(s)}}{|x|^s}{\rm d}x\bigg)^{\tfrac{1}{2^*(s)}}< \varepsilon\no
		\eea
		Thus $\frac{|u_n|^{2^*(s)-2}u_nv}{|x|^s}$ is uniformly integrable in $B_R$. Therefore,  using Vitali's convergence theorem, we can pass to the limit in the 1st integral on RHS of \eqref{24-5-2}.
		
		To estimate the integral now on $B_R^c$, we first set $v_n=u_n-u_0$.
		Then $v_n\rightharpoonup 0$ in $\dot{H}^1(\Rn).$ It is not difficult to see that for every $\varepsilon >0$ there exists
		$C_\varepsilon>0$ such that
		$$\bigg{|}\frac{|v_n+u_0|^{2^*(s)-2}(v_n+u_0)}{|x|^s}-\frac{|u_0|^{2^*(s)-2}u_0}{|x|^s} \bigg{|} <\varepsilon \frac{|v_n|^{2^*(s)-1}}{|x|^s}+C_\varepsilon\frac{|u_0|^{2^*(s)-1}}{|x|^s}.$$
		Therefore,
		
		\begin{align*}
			&\bigg{|}\int_{B_R^c}\bigg\{\frac{|u_n|^{2^*(s)-2}u_n}{|x|^s}-\frac{|u_0|^{2^*(s)-2}u_0}{|x|^s}\bigg\}v\;{\rm d}x\bigg{|}\\
			&\leq \bigg[\varepsilon \int_{B_R^c}\frac{|v_n|^{2^*(s)-1}|v|}{|x|^s}\;{\rm d}x+C_\varepsilon\int_{B_R^c}\frac{|u_0|^{2^*(s)-1}|v|}{|x|^s}\,{\rm d}x\bigg] \\
			&\leq \bigg[\varepsilon \bigg(\int_{B_R^c}\frac{|v_n|^{2^*(s)}}{|x|^s}\;{\rm d}x\bigg)^{\tfrac{2^*(s)-1}{2^*(s)}}\bigg(\int_{B_R^c}
			\frac{|v|^{2^*(s)}}{|x|^s}\,{\rm d}x\bigg)^{\tfrac{1}{2^*(s)}}\\	&\qquad\qquad\qquad+C_\varepsilon\bigg(\int_{B_R^c}\frac{|u_0|^{2^*(s)}}{|x|^s}\,{\rm d}x\bigg)^{\tfrac{2^*(s)-1}{2^*(s)}}\bigg(\int_{B_R^c}\frac{|v|^{2^*(s)}}{|x|^s}\,{\rm d}x\bigg)^{\tfrac{1}{2^*(s)}}\bigg]\\
			&\leq C\bigg[\varepsilon \|v_n\|_{\gamma}^{2^*(s)-1}\bigg(\int_{B_R^c}
			\frac{|v|^{2^*(s)}}{|x|^s}\,{\rm d}x\bigg)^{\tfrac{1}{2^*(s)}}
			+C_\varepsilon\|u_0\|_{\gamma}^{2^*(s)-1}\bigg(\int_{B_R^c}
			\frac{|v|^{2^*(s)}}{|x|^s}\,{\rm d}x\bigg)^{\tfrac{1}{2^*(s)}}\bigg].
		\end{align*}
		
		Since $\left\{\|v_n\|_{\ga}\right\}_n$ is uniformly bounded and $\frac{|v|^{2^*(s)}}{|x|^s}\in L^1(\Rn)$, given $\var>0$,  we can choose $R>0$ large enough such that	$$\bigg{|}\int_{B_R^c}\bigg\{\frac{|u_n|^{2^*(s)-2}u_n}{|x|^s}-\frac{|u_0|^{2^*(s)-2} u_0}{|x|^s}\bigg\}v\;{\rm d}x\bigg{|}<\var.$$
		This completes the proof of claim 1.
		
		Hence \eqref{24-5-1} yields that $u_0$ is a solution of~\eqref{MainEq}.
		
		\medskip
		
		\noindent\underline{\bf Step 3:} Here we show that $\{u_n-u_0\}_n$ is a $(PS)$ sequence for $I_{\ga,s}$ at the level $\ba-I_{\ga,s}(u_0)$.
		To see this, first, we observe that as $n\to\infty$
		$$\|u_n-u_0\|_{\gamma}^2 = \|u_n\|_{\gamma}^2-\|u_0\|_{\gamma}^2+o(1),$$
		and by the Br\'ezis-Lieb lemma as $n\to\infty$
		$$\int_{\Rn}\frac{|u_n-u_0|^{2^*(s)}}{|x|^s}{\rm d}x = \int_{\Rn}\frac{|u_n|^{2^*(s)}}{|x|^s}\;{\rm d}x - \int_{\Rn}\frac{|u_0|^{2^*(s)}}{|x|^s}\;{\rm d}x+o(1).$$
		Therefore, as $n\to\infty,$
		\begin{align*}
			I_{\ga,s} (u_n-u_0) &= \frac{1}{2}\|u_n-u_0\|_{\gamma}^2-\frac{1}{2^*(s)}\int_{\Rn}\frac{|u_n-u_0|^{2^*(s)}}{|x|^s}\;{\rm d}x\\
			&= \frac{1}{2}\|u_n\|_{\gamma}^2-\frac{1}{2^*(s)}\int_{\Rn}\frac{|u_n|^{2^*(s)}}{|x|^s}\;{\rm d}x\\
			&\,-\bigg\{\frac{1}{2}\|u_0\|_{\gamma}^2-\frac{1}{2^*(s)}\int_{\Rn}\frac{|u_0|^{2^*(s)}}{|x|^s}\;{\rm d}x\bigg\}+o(1)\\
			&= I_{\ga,s}(u_n)-I_{\ga,s}(u_0)+o(1)\\
			&\To \ba-I_{\ga,s}(u_0).
		\end{align*}
		
		Further, as $\prescript{}{H^{-1}}{\big\langle}I_{\ga,s}'(u_0), v{\big\rangle}_{\dot H^1} =0$ for all $v\in\dot{H}^(\Rn)$, we  obtain
		\begin{align}\lab{24-5-3}
			\prescript{}{H^{-1}}{\big\langle}I_{\ga,s}'(u_n-u_0), v{\big\rangle}_{\dot H^1}&=\langle u_n-u_0,v\rangle_{\gamma}-\int_{\Rn}\frac{|u_n- u_0|^{2^*(s)-2}(u_n-u_0)v}{|x|^s}\;{\rm d}x\no\\
			&=\langle u_n,v\rangle_{\gamma}-\int_{\Rn}\frac{|u_n|^{2^*(s)-2}u_nv}{|x|^s}\;{\rm d}x\no\\
			&-\bigg(\langle u_0,v\rangle_{\gamma}-\int_{\Rn}\frac{|u_0|^{2^*(s)-2}u_0 v}{|x|^s}\;{\rm d}x\bigg)\\
			&+\int_{\Rn}\bigg\{\frac{|u_n|^{2^*(s)-2}u_n}{|x|^s}-\frac{|u_0|^{2^*(s)-2}u_0}{|x|^s}\no\\
			&\qquad\qquad\qquad-\frac{|u_n-u_0|^{2^*(s)-2}(u_n-u_0)}{|x|^s}\bigg\}v {\rm d}x\no\\
			&=o(1)+\int_{\Rn}\bigg\{\frac{|u_n|^{2^*(s)-2}u_n}{|x|^s}-\frac{|u_0|^{2^*(s)-2} u_0}{|x|^s}\no\\
			&\qquad\qquad\qquad\qquad-\frac{|u_n-u_0|^{2^*(s)-2}(u_n-u_0)}{|x|^s}\bigg\}v {\rm d}x.\no
		\end{align}
		
		We observe that
		\begin{align*}%\lab{12-10-2}
			&\bigg|\left\{|u_n|^{2^*(s)-2}u_n-|u_0|^{2^*(s)-2}
			u_0-|u_n-u_0|^{2^*(s)-2}(u_n-u_0)\right\}\bigg|\\
			&\hspace{4cm}\leq C\bigg(|u_n-u_0|^{2^*(s)-2}|u_0|+|u_0|^{2^*(s)-2}|u_n-u_0|\bigg).
		\end{align*}
		Therefore, following the same method as in the proof of Claim~1 in Step~2,
		we show  that as $n\to\infty$
		\be\lab{25-5-8}\displaystyle\int_{\Rn}\bigg\{\frac{|u_n|^{2^*(s)-2}u_n}{|x|^s}-\frac{|u_0|^{2^*(s)-2} u_0}{|x|^s}-\frac{|u_n-u_0|^{2^*(s)-2}(u_n-u_0)}{|x|^s}\bigg\}v{\rm d}x=o(1)\ee
		for all  $v\in\dot{H}^1(\Rn)$.
		Plugging this back into \eqref{24-5-3}, we complete the proof of
		Step~3.
		\medskip
		
		\noindent\underline{\bf Step 4:} Define $v_n:=u_n-u_0.$
		Then $v_n\rightharpoonup 0$ in $\dot{H}^1(\Rn)$ and by Step~3, $\{v_n\}_n$ is a $(PS)$ sequence for $I_{\ga,s}$ at the level $\ba-I_{\ga,s}(u_0).$
		Thus,
		\be\lab{10-9-22}
		\sup_{n\in\mathbb{N}}\|v_n\|_{\gamma}\leq C \quad\mbox{and}\quad
		\langle v_n, \va\rangle_{\ga}=\int_{\Rn}\frac{|v_n|^{2^*(s)-2}v_n\va}{|x|^s} \, {\rm d}x+o(1)\quad\text{for all }\phi\in \dot{H}^1(\Rn).\ee
		Therefore,	
		$\|v_n\|_{\gamma}^2=\int_{\Rn}\frac{|v_n|^{2^*(s)}}{|x|^s}\;{\rm d}x+o(1)$. Thus, if $\int_{\Rn}\frac{|v_n|^{2^*(s)}}{|x|^s}\;{\rm d}x\longrightarrow 0$, then we are done where $m=0$	and the $(PS)$ sequence $\{u_n\}_n$ admits a strongly convergent subsequence.
		
		If not, let $0<\delta <\mu_{\gamma,s}^{\frac{N-s}{2-s}}$ such that
		$$\limsup_{n\to\infty}\int_{\Rn}\frac{|v_n|^{2^*(s)}}{|x|^s}\;{\rm d}x>\delta.$$
		Up to a subsequence, let $R_n>0$ be such that
		$$\int_{B_{R_n}}\frac{|v_n|^{2^*(s)}}{|x|^s}\;{\rm d}x = \delta$$
		and $R_n$ being minimal with this property. Define
		$$w_n(x):= R_n^{\frac{N-2}{2}}v_n(R_nx).$$ Therefore, $\|w_n\|_{\gamma}=\|v_n\|_{\gamma}$ and
		\be\label{PS3}
		\delta = \int_{B_{R_n}}\frac{|v_n|^{2^*(s)}}{|x|^s}\;{\rm d}x = \int_{B_1}\frac{|w_n|^{2^*(s)}}{|x|^s}\;{\rm d}x.
		\ee
		Therefore, up to a subsequence, there exists $w\in\dot{H}^1(\Rn)$ such that
		$$w_n \deb w\;\mbox{in }\dot{H}^1(\Rn) \quad\mbox{and}\quad
		w_n\to w \mbox{ a.e. in }\Rn.$$
		
		\medskip
		Let us now distinguish two cases $w\neq 0$ and $w=0$.
		\vspace{2mm}
		
		\noindent\underline{\bf Step 5:} Assume that $w\neq 0.$
		
		Since, $w_n\rightharpoonup w\neq 0$ and $v_n\rightharpoonup 0$ in $\dot{H}^1(\Rn)$, it follows that $R_n\to 0$ as $n\to\infty$. Next, we show that $w$ is a solution of \eqref{MainEq}.
		Indeed, thanks to \eqref{10-9-22}, for any  $\phi\in C_c^{\infty}(\Rn)$
		\begin{align}\lab{ReSol}
			\langle w,\phi\rangle_{\gamma} &= \lim_{n\to\infty}\langle w_n,\phi\rangle_{\gamma}\no\\
			&=\lim_{n\to\infty} \int_{\Rn}\nabla w_n(x)\cdot\nabla \phi(x)\,{\rm d}x-\gamma\int_{\Rn}\frac{w_n\phi}{|x|^{2}}\;{\rm d}x\no\\
			&=\lim_{n\to\infty} \left[\int_{\Rn}R_n^{\tfrac{N}{2}}\nabla v_n(R_nx)\cdot\nabla \phi(x)\,{\rm d}x-\gamma\int_{\Rn}\frac{R_n^{\frac{N-2}{2}}v_n(R_nx)\phi(x)}{|x|^{2}}\;{\rm d}x\right]\\
			&=\lim_{n\to\infty} \left[\int_{\Rn}R_n^{-\tfrac{N}{2}}\nabla v_n(x)\cdot\nabla \phi(\tfrac{x}{R_n})\,{\rm d}x-\gamma\int_{\Rn}\frac{R_n^{-\frac{N-2}{2}}v_n(x)\phi(\frac{x}{R_n})}{|x|^{2}}\;{\rm d}x\right]\no\\ &=\lim_{n\to\infty}\int_{\Rn}\!\!\!\frac{|v_n|^{2^*(s)-2}v_n}{|x|^s}R_n^{-\frac{N-2}{2}}\phi\big(\tfrac{x}{R_n}\big)\;{\rm d}x=\lim_{n\to\infty}\int_{\Rn}\!\!\!\frac{|w_n|^{2^*(s)-2}w_n}{|x|^s}\phi(x)\;{\rm d}x.\no
		\end{align}
		
		Clearly $\frac{|w_n|^{2^*(s)-2}w_n}{|x|^s}\phi\to \frac{|w|^{2^*(s)-2}w}{|x|^s}\phi$ a.e. in $\Rn$,
		since $w_n\to w$ a.e. in $\Rn$. Further, arguing as in the proof of Claim~1 in Step~2, we have
		$\frac{|w_n|^{2^*(s)-2}w_n}{|x|^s}\phi$ is uniformly integrable. Therefore,
		as $\phi$ has compact support, using Vitali's convergence theorem we obtain
		\be\lab{ReSol1}\lim_{n\to\infty}\int_{\Rn}\frac{|w_n|^{2^*(s)-2}w_n}{|x|^s}\phi(x)\;{\rm d}x = \int_{\Rn}\frac{|w|^{2^*(s)-2}w\phi}{|x|^s}{\rm d}x.\ee
		Combining \eqref{ReSol1} along with \eqref{ReSol}, we conclude
		that $w$ is a solution of \eqref{MainEq}.
		
		Define	 $$z_n(x):=v_n(x)-R_n^{-\frac{N-2}{2}}w(\tfrac{x}{R_n}).$$
		
		\noindent{\bf Claim 2:} $\{z_n\}_n$ is a $(PS)$ sequence for $I_{\ga,s}$ at the level $\ba-I_{\ga,s}(u_0)- I_{\ga,s}(w).$
		
		To prove the claim, set $$\tilde{z}_n(x):=R_n^{\frac{N-2}{2}}z_n(R_nx).$$ Then
		$$ \tilde{z}_n(x)= w_n(x)-w(x) \quad\mbox{and}\quad  \|\tilde{z}_n\|_{\gamma}=\|w_n-w\|_{\gamma}=\|z_n\|_{\gamma}.$$
		From Br\'ezis-Lieb lemma,
		\begin{align*}
			\int_{\Rn}\frac{|w_n(x)|^{2^*(s)}}{|x|^s}\;{\rm d}x-\int_{\Rn}\frac{|w|^{2^*(s)}}{|x|^s}\;{\rm d}x
			% 		&=\int_{\Rn}\frac{\big|w_n-w\big|^{2^*(s)}}{|x|^s}\;{\rm d}x+o(1)\\
			&=\int_{\Rn}\frac{|w_n-w|^{2^*(s)}}{|x|^s}\;{\rm d}x+o(1).
		\end{align*}
		
		Therefore, using the above relations, as $n\to\infty$
		\begin{align*}
			I_{\ga,s}(z_n)&=\frac{1}{2}\|z_n\|_{\gamma}^2-\frac{1}{2^*(s)}\int_{\Rn}\frac{|z_n|^{2^*(s)}}{|x|^s}\;{\rm d}x\\
			&=\frac{1}{2}\|w_n-w\|_{\gamma}^2-\frac{1}{2^*(s)}\int_{\Rn}\frac{|w_n-w|^{2^*(s)}}{|x|^s}\;{\rm d}x\\
			&=\frac{1}{2}\big(\|w_n\|_{\gamma}^2-\|w\|_{\gamma}^2\big)-\frac{1}{2^*(s)}\int_{\Rn}\frac{|w_n(x)|^{2^*(s)}}{|x|^s}\;{\rm d}x+\frac{1}{2^*(s)}\int_{\Rn}\frac{|w|^{2^*(s)}}{|x|^s}\;{\rm d}x+o(1)\\
			&=\frac{1}{2}\|v_n\|_{\gamma}^2-\frac{1}{2^*(s)}\int_{\Rn}\frac{|v_n(x)|^{2^*(s)}}{|x|^s}\;{\rm d}x-\bigg(\frac{1}{2}\|w\|_{\gamma}^2-\frac{1}{2^*(s)}\int_{\Rn}\frac{|w|^{2^*(s)}}{|x|^s}\;{\rm d}x\bigg)+o(1)\\
			&=I_{\ga,s}(v_n)- I_{\ga,s}(w)+o(1)\\
			&=\ba-I_{\ga,s}(u_0)- I_{\ga,s}(w)+o(1).
		\end{align*}
		Next, let $\phi\in C^\infty_c(\Rn)$ be arbitrary and set $\phi_n(x):=R_n^{\frac{N-2}{2}}\phi(R_nx)$. This in turn implies that
		$\|\phi_n\|_{\gamma}=\|\phi\|_{\gamma}$ and $\phi_n\rightharpoonup 0$ in $\dot{H}^1(\Rn)$.
		
		Therefore,
		\begin{align}\lab{DIz}
			\prescript{}{\dot{H}^{-1}}{\big\langle}I'_{\ga,s}(z_n), \phi{\big\rangle}_{\dot \dot{H}^1} &=\langle z_n,\phi\rangle_{\gamma}-\int_{\Rn}\frac{|z_n|^{2^*(s)-2}z_n\phi}{|x|^s}\;{\rm d}x\no\\
			&= \langle \tilde{z}_n,\phi_n\rangle_{\gamma}-\int_{\Rn}\frac{|\tilde{z}_n|^{2^*(s)-2}\tilde{z}_n\phi_n}{|x|^s}\;{\rm d}x\no\\
			&=\langle w_n-w,\phi_n\rangle_{\gamma}-\int_{\Rn}\frac{|w_n-w|^{2^*(s)-2}(w_n-w)\phi_n}{|x|^s}\;{\rm d}x\no\\
			&= \langle w_n,\phi_n\rangle_{\gamma}-\int_{\Rn}\frac{|w_n|^{2^*(s)-2}w_n\phi_n}{|x|^s}\;{\rm d}x-\bigg(\langle w,\phi_n\rangle_{\gamma} -\int_{\Rn}\frac{|w|^{2^*(s)-2}w\phi_n}{|x|^s}\;{\rm d}x\bigg)\\
			% 		&+\int_{\Rn}(K(R_nx)-1)\frac{|w|^{2_s^*(t)-2}w\phi_n}{|x|^t}{\rm d}x\no\\
			&+\int_{\Rn}\!\!\!\bigg(\frac{|w_n|^{2^*(s)-2}w_n-|w|^{2^*(s)-2}w-|w_n-w|^{2^*(s)-2}(w_n-w)}{|x|^s}\bigg)\phi_n {\rm d}x\no\\
			&=\langle v_n,\phi\rangle_{\gamma}-\int_{\Rn}\frac{|v_n|^{2^*(s)-2}v_n\phi}{|x|^s}\;{\rm d}x-\prescript{}{\dot{H}^{-1}}{\big\langle} I_{\ga,s}{'}(w), \phi_n{\big\rangle}_{\dot H^1}+I^1_n\no\\
			&= \prescript{}{\dot{H}^{-1}}{\big\langle}I'_{\ga,s}(v_n), \phi{\big\rangle}_{\dot H^1}-0+I^1_n=o(1)+I^1_n.\no
		\end{align}
		
		Now, we aim to show that  $$I^1_n:=\int_{\Rn}\bigg\{\frac{|w_n|^{2^*(s)-2}w_n-|w|^{2^*(s)-2}w-|w_n-w|^{2^*(s)-2}(w_n-w)}{|x|^s}\bigg\}\phi_n {\rm d}x=o(1).$$
		Indeed, this follows as in the proof of \eqref{25-5-8}, since $\lim_{n\to\infty}\int_{\Rn}\frac{|\phi_n|^{2^*(s)}}{|x|^s}\;{\rm d}x=\int_{\Rn}\frac{|\phi|^{2^*(s)}}{|x|^s}\;{\rm d}x<\infty$.
		Hence, from \eqref{DIz} we conclude the proof of Claim 2.
		
		\medskip
		
		\noindent\underline{\bf Step 6:}  Assume that $w= 0$. We will show that this case can not occur unless $s=0$.
		\vspace{2mm}
		
		Let $\va\in C^\infty_c\big(B_1\big)$, with $0\leq\va\leq 1$.
		% Set $\psi_n(x) :=[\va(\tfrac{x}{R_n})]^2v_n(x)$. Clearly $(\psi_n)_n$is a bounded sequence in~$\dot{H}^1(\Rn)$.
		Using Rellich-Kondrachov's compactness theorem we have, $w_n\to 0$ in $L^2_{loc}(\Rn)$. Thus, we have,
		
		\bea
		\int_{\Rn}|\nabla (\va w_n)|^2\,{\rm d}x &=& \int_{\Rn}w_n^2 |\nabla \va|^2\,{\rm d}x + \int_{\Rn} \nabla w_n\cdot (\va^2\nabla w_n+2w_n \va\nabla \va)\,{\rm d}x\no\\
		&=&\int_{\Rn} w_n^2 |\nabla \va|^2\,{\rm d}x + \int_{\Rn} \nabla w_n\cdot \nabla(\va^2w_n)\,{\rm d}x\no\\
		&=&\prescript{}{\dot{H}^{-1}}{\big\langle}I'_{\ga,s}(w_n),\va^2w_n{\big\rangle}_{\dot{H}^1} + \ga\int_{\Rn}\frac{|w_n|^2\va^2}{|x|^2}\,{\rm d}x+\int_{\Rn}\frac{|w_n|^{2^*(s)}\va^2}{|x|^s}\,{\rm d}x +o(1)\no\\
		&\leq& \|I'_{\ga,s}(w_n)\|_{\dot{H}^{-1}(\Rn)}\|\va^2w_n\|_{\dot{H}^1(\Rn)}+ \frac{4\ga}{(N-2)^2}\int_{\Rn}|\nabla(\va w_n)|^2\,{\rm d}x\no\\
		&&\qquad\qquad +\mu_{\ga,s}(\Rn)^{-1}\left(\int_{\Rn}\frac{|w_n|^{2^*(s)}}{|x|^s}\,{\rm d}x\right)^{\tfrac{2-s}{N-s}}\left(\int_{\Rn}|\nabla(\va w_n)|^2\,{\rm d}x\right)+o(1)\no\\
		&\leq&(1-\tau)\int_{\Rn}|\nabla (\va w_n)|^2\,{\rm d}x+o(1),\quad\text{for some $\tau>0$, by the choice of }\delta.\no
		\eea
		
		In particular, from the Hardy-Sobolev inequality, we deduce that, for each $0<r<1,$
		
		$$\int_{B_r}\frac{|w_n|^{2^*(s)}}{|x|^s}\,{\rm d}x\to 0,\,\,\text{as }n\to\infty.$$
		
		But this contradicts \eqref{PS3} when $s>0$. Therefore, $w=0$ can not happen in our case $s\in (0,2).$

		\medspace
		In both cases, starting from a $(PS)$ sequence $\{v_n\}_n$ of $I_{\ga,s}$ we have singled out another $(PS)$ sequence $\{z_n\}_n$ at a strictly lower level, with a fixed minimum amount of decrease. Therefore, arguing recursively and using the fact that $\sup_{n\in\mathbb{N}}\|u_n\|_{\ga}<\infty,$ we infer that the process should terminate after finite steps and the last $(PS)$ should converge strongly to $0$ in $\dot{H}^1(\Rn)$. This completes the proof.
	\end{proof}
	\medspace
    As a consequence of the above Theorem, we have the following qualitative stability result:
	\begin{theorem}\label{SDHS}
		Let $N\geq 3$ and $\nu\in\mathbb{N}$. Let $\{u_k\}_{k\in\mathbb{N}}\subseteq \dot{H}^1(\Rn)$ be a sequence of nonnegative functions such that $(\nu-\frac{1}{2})\mu_{\ga,s}(\Rn)^{\tfrac{N-s}{2-s}}\leq \|u_k\|_{\ga}^2\leq (\nu+\frac{1}{2})\mu_{\ga,s}(\Rn)^{\frac{N-s}{2-s}}$ with $\mu_{\ga,s}(\Rn)$ defined as in \eqref{BCHSI} and assume that $\Ga(u_k)\to 0$ as $k\to\infty.$
		
		Then there exists a sequence $(\la_k^1,\cdots,\la_k^\nu)_{k\in\mathbb{N}}$ of $\nu-$tuples of positive real numbers such that 
		$$\bigg{\|}u_k-\sum_{i=1}^{\nu}U_{\ga,s}^{\la_k^i}\bigg{\|}_{\ga}\to 0\quad \text{as }k\to\infty.$$
	\end{theorem}

	\medskip

	\section{Stability for the equation}\label{SCFMHS}
	
		In this section, we investigate the validity of a sharp quantitative version of the stability of critical points for the functional whose Euler-Lagrange equation is the Hardy-Sobolev equation \eqref{MainEq}. Define the quantity 
	\be\label{delu}
	\Gamma (u) :=\Bigg{\|}\De u+\ga \frac{u}{|x|^{2}} + \frac{u^{2^*(s)-1}}{|x|^s}\Bigg{\|}_{\dot{H}^{-1}(\Rn)}. 
	\ee
	\begin{definition}[Interaction between bubbles]
		Let $U_{\ga,s}^1=U_{\ga,s}^{\la_1},\cdots ,\,U_{\ga,s}^{\nu}=U_{\ga,s}^{\la_{\nu}}$ be a family of bubbles. We say that this family is $\de$-interacting for some $\de>0$ if
		\be\label{ibb}
		\min_{1\leq i,j \leq \nu}\left(\frac{\la_i}{\la_j},\,\frac{\la_j}{\la_i}\right)\leq \de
		\ee
		If together with the family we also have some positive coefficients $\al_1,\cdots,\,\al_{\nu}\in\mathbb{R},$ we say that the family with the coefficients is $\de$-interacting if \eqref{ibb} holds and in addition we have
		$$\max_{1\leq i\leq \nu}|\al_i -1|\leq \de$$
	\end{definition}
	\begin{theorem}
		Given $N\geq 3$, there exists $C_0(N,\ga,s)>0$ with the following property. Let $u\in C^{\infty}(\Rn\setminus\{0\})\cap \dot{H}^1(\Rn)$ be a nonnegative function satisfying
		%		$$K_0(u)=1,\qquad \text{ and }\qquad$$
		$$ \frac{1}{2}\mu_{\ga,s}(\Rn)^{\tfrac{N-s}{2-s}}\leq \|u\|_{\ga}^2 \leq \frac{3}{2}\mu_{\ga,s}(\Rn)^{\tfrac{N-s}{2-s}}.$$
		Then there exists $\la>0$ such that
		$$u = U_{\ga,s}^{\la}+\rho$$
		where $$\|\rho\|_{\ga}\lesssim_{N,\ga,s}\Gamma (u).$$
	\end{theorem}%\right) 
	
	\begin{proof}
		%{\color{red}It is enough to prove the theorem for $u:\Rn\to \R$ satisfying 
			%	$$u\in C^{\infty}(\Rn\setminus \{0\})\cap \dot{H}^1(\Rn),\,\,u>0\qquad\text{on }\Rn$$
			%	and such that 
			%	\be\label{OBC1}
			%	\|u\|_{\ga}\leq \frac{3}{2}\mu_{\ga,s}(\Rn)^{\tfrac{N-s}{2-s}},\qquad \Ga=\Ga(u)\leq \Ga_0
			%	\ee
			%	for a suitably small constant $\Ga_0=\Ga_0(N)$.
			
			%	Indeed, for $\Ga(u)> \Ga_0,$ the theorem is trivially true simply by choosing $r=1,$ setting,
			%	$$\rho = u-W$$
			%	and then simply choosing $C_0=C_0(N)$ large enough.}
		{\bf Step I:} First we approximate $u$ not only with one bubble but a scalar multiple of the bubble, namely we consider the manifold
		$$\tM:=\left\{\al U_{\ga,s}^\la\,: U_{\ga,s}^{\la}(x)=\la^{\tfrac{N-2}{2}}U_{\ga,s}(\la x),\,\la>0,\,\al\in\R\setminus\{0\}\right\}.$$
		Later we shall recover the result for one bubble.
		
		Up to enlarging the constant $C_0>0$ we can assume that $\Ga(u):=\|\De u+\frac{\ga}{|x|^2}u+\frac{u^{2^*(s)-1}}{|x|^s}\|_{\dot{H}^{-1}}\leq \delta_0$. 
		By Theorem \ref{SDHS} for a constant $\varepsilon_0>0$ (to be fixed later on), we can choose $\delta_0>0$ depending on $\varepsilon_0$ in such a way that there exists $(\al,\la)\in \mathbb{R}\times \mathbb{R}_{+}$ such that
		\bea
		\|u-\al U_{\ga,s}^{\la}\|_{\ga}&\leq& \varepsilon_0,\label{STOB1}\\
		|\al-1|&\leq& \varepsilon_0\label{STOB2}.
		\eea
		Without loss of generality, we may choose $(\al,r)\in\mathbb{R}\times\mathbb{R}_{+}$ such that $\sigma:=\al U_{\ga,s}^{\la}$ be the bubble closest to $u$ in the $\|\cdot\|_{\ga}$-norm, that is,
		\be\label{OPM}
		\|u-\sigma\|_{\ga} = \min_{\(\tilde{\al},\tilde{\la}\)\in\R\times \R_{+}}\|u-\tilde{\al}U_{\ga,s}^{\tilde{\la}}\|_{\ga}.
		\ee
		
		Let $\rho :=u-\sigma$ be the difference between the function $u$ and the orthogonal projection of $u$ onto $\tM$, namely, $\sigma$. Therefore, from \eqref{STOB1} we have, 
		\be\label{STOB3}
		\|\rho\|_{\ga}\leq \varepsilon_0.
		\ee
		\medspace
		Since $\sigma\in\dot{H}^1(\Rn)$ minimizes the $\dot{H}^1$-distance (with respect to $\|\cdot\|_{\ga}$-norm) from u to $\tM$,  $\rho$ is $\dot{H}^1$-orthogonal to $T_{\sigma}\tM$. That is,
		\bea
		\int_{\Rn}\nabla \rho\cdot\nabla U_{\ga,s}^{\la}\,{\rm d}x-\ga\int_{\Rn}\frac{\rho U_{\ga,s}^{\la}}{|x|^2}\,{\rm d}x&=&0\label{OR1}\\
		\int_{\Rn}\nabla \rho\cdot\nabla V\,{\rm d}x-\ga\int_{\Rn}\frac{\rho V}{|x|^2}\,{\rm d}x&=&0\label{OR2}
		\eea
		where $V=\frac{{\rm d}}{{\rm d}t}U_{\ga,s}^{t}\bigg{|}_{t=\la}$.
		%	{\bf{\underline{Step I}}:} We assume the Struwe's decomposition (\cite{S}) for \eqref{HSE1}.
		Since the functions $U_{\ga,s}^{\la},\,V=\frac{{\rm d}}{{\rm d}t}U_{\ga,s}^t\bigg{|}_{t=\la}$ are eigenfunctions for $\mathcal{L}_{\ga,s,\la}:= \left(\frac{-\De-\frac{\ga}{|x|^2}}{\tfrac{\big(U_{\ga,s}^{\la}\big)^{2^*(s)-2}}{|x|^s}}\right)$ (see Lemma \ref{STLgar}),the above-mentioned orthogonality conditions are equivalent to 
		\bea
		\int_{\Rn}\rho \frac{\big(U_{\ga,s}^{\la}\big)^{2^*(s)-1}}{|x|^s}\,{\rm d}x &=& 0\label{OR3}\\
		\int_{\Rn}\rho V \frac{\big(U_{\ga,s}^{\la}\big)^{2^*(s)-2}}{|x|^s}\,{\rm d}x &=& 0\label{OR4}
		\eea
		Our main target is to prove that $\|\rho\|_{\ga}$ is controlled by $\Ga(u)$. To achieve our goal we start by testing $\De u +\frac{\ga}{|x|^2}u+\frac{u^{2^*(s)-1}}{|x|^s}$ against $\rho$: exploiting the orthogonal conditions \eqref{OR1}-\eqref{OR4} we get,
		\bea
		\|\rho\|_{\ga}^2 = \langle u,\rho\rangle_{\ga} &=& \int_{\Rn}\rho \frac{u^{2^*(s)-1}}{|x|^s}\,{\rm d}x - \int_{\Rn}\rho \left(\De u +\frac{\ga}{|x|^2}u+\frac{u^{2^*(s)-1}}{|x|^s}\right)\,{\rm d}x\no\\
		&\leq& \int_{\Rn}\rho \frac{u^{2^*(s)-1}}{|x|^s}\,{\rm d}x + \|\rho\|_{\ga}\bigg{\|}\left(\De u +\frac{\ga}{|x|^2}u+\frac{u^{2^*(s)-1}}{|x|^s}\right)\bigg{\|}_{\dot{H}^{-1}}\label{OBC}
		\eea
		To control the first term in the RHS of the above integral we consider the following estimate
		\be\label{EE}
		\bigg{|}(a+b)|a+b|^{p-1}+a|a|^{p-1}\bigg{|}\leq p|a|^{p-1}|b|+C_N\left(|a|^{p-2}|b|^2+|b|^p\right)
		\ee
		The above inequality holds without the term $|a|^{p-2}|b|^2$ whenever $p-2<0$.
		We put $p=2^*(s)-1=\frac{N+2-2s}{N-2},\,a=\frac{\sigma}{|x|^{\tfrac{s}{p}}},\,b:=\frac{\rho}{|x|^{\tfrac{s}{p}}},$. Then, $p-2\leq 0$ for every $N\geq 6$. In order to get $p>2$, we must have, $N<6-2s$. Therefore, $p>2$ can happen for $N=3,\,4,\,5$ (depending on $s\in (0,2]$).
		\medspace
		
		Thus we get, for every $N\geq 6$, 
		\be\no
		\bigg{|}\frac{u^{p}}{|x|^s}- \frac{\sigma^{p}}{|x|^s}\bigg{|}\leq p \frac{\sigma^{p-1}\rho}{|x|^s}+C_N \frac{|\rho|^p}{|x|^s}
		\ee
		Therefore, recalling \eqref{STOB1} and using H\"older's inequality together with Hardy-Sobolev inequality we get
		\bea
		\int_{\Rn}\frac{u^{2^*(s)-1}\rho}{|x|^s}\,{\rm d}x &\leq& p\int_{\Rn}\frac{|\sigma|^{2^*(s)-2}\rho^2}{|x|^s}\,{\rm d}x + C_N \int_{\Rn}\frac{|\rho|^{2^*(s)}}{|x|^s}\,{\rm d}x\label{STOB4}
		%	&\leq& p \left(\int_{\Rn}\frac{|\sigma|^{2^*(s)}}{|x|^s}\,{\rm d}x\right)^{\tfrac{2^*(s)-2}{2^*(s)}}\left(\int_{\Rn}\frac{|\rho|^{2^*(s)}}{|x|^s}\,{\rm d}x\right)^{\tfrac{2}{2^*(s)}}+C_N\mu_{\ga,s}(\Rn)^{-\tfrac{2^*(s)}{2}}\|\rho\|_{\ga}^{2^*(s)}\no\\
		%	&\preceq& \|\rho\|_{\ga}^{2}+\|\rho\|_{\ga}^{2^*(s)}\no
		\eea
		By Spectral analysis argument for the operator $\mathcal{L}_{\ga,s,\la}$ (see Lemma \ref{Lgar}) and using the fact that $\rho\perp T\tM_{U_{\ga,s}^{\la}}$ we get, 
		\be\no
		\|\rho\|_{\ga}^2\geq \Lambda \int_{\Rn}\frac{\big(U_{\ga,s}^{\la}\big)^{2^*(s)-2}\rho^2}{|x|^s}\,{\rm d}x,
		\ee
		where $\Lambda(N) >2^*(s)-1$.
		Thus,
		\be\label{STOB5}
		\int_{\Rn}\frac{|\sigma|^{2^*(s)-2}\rho^2}{|x|^s}\,{\rm d}x = |\al|^{2^*(s)-2}\int_{\Rn}\frac{|U_{\ga,s}^{\la}|^{2^*(s)-2}\rho^2}{|x|^s}\,{\rm d}x\leq \frac{|\al|^{2^*(s)-2}}{\Lambda}\|\rho\|_{\ga}^2.
		\ee
		Using this, from \eqref{STOB4} we get,
		\be\label{STOB6}
		\int_{\Rn}\frac{u^{2^*(s)-1}\rho}{|x|^s}\,{\rm d}x \leq \frac{|\al|^{2^*(s)-2}p}{\Lambda}\|\rho\|_{\ga}^2+C_N\mu_{\ga,s}(\Rn)^{-\tfrac{2^*(s)}{2}}\|\rho\|_{\ga}^{2^*(s)}
		\ee
		Hence recalling \eqref{OBC} we get,
		\be\no
		\left(1-\frac{|\al|^{2^*(s)-2}p}{\Lambda}\right)\|\rho\|_{\ga}^2\leq C_{N,\ga}\|\rho\|_{\ga}^{2^*(s)}+\|\rho\|_{\ga}\Gamma (u).
		\ee
		Since, we may assume that, $\|\rho\|_{\ga}\ll 1,$ we can say that the last inequality implies (for $N\geq 6$)
		\be\label{OBC1}
		\|\rho\|_{\ga} \preceq \Ga (u).
		\ee
		Whenever $p>2$ \big{(}can happen for $N=3,\,4,\,5$, depending upon $s\in (0,2]$\big{)}, we get,
		\be\no
		\bigg{|}\frac{u^{p}}{|x|^s}- \frac{\sigma^{p}}{|x|^s}\bigg{|}\leq p \frac{\sigma^{p-1}\rho}{|x|^s}+C_N \left(\frac{|\rho|^p}{|x|^s}+\frac{|\sigma|^{p-2}\rho^2}{|x|^s}\right).
		\ee
		Now, since $p>2$ using H\"older's inequality and Hardy-Sobolev inequality we must have,
		\be\no
		\int_{\Rn}\frac{|\sigma|^{2^*(s)-3}\rho^3}{|x|^s}\,{\rm d}x \leq \left(\int_{\Rn}\frac{|\sigma|^{2^*(s)}}{|x|^s}\,{\rm d}x \right)^{\tfrac{2^*(s)-3}{2^*(s)}}\left(\int_{\Rn}\frac{|\rho|^{2^*(s)}}{|x|^s}\,{\rm d}x \right)^{\tfrac{3}{2^*(s)}}\lesssim_{N,\ga,s} \|\rho\|_{\ga}^{3}.
		\ee
		Therefore, following similar arguments in this case, we get, 
		\be\no
		\|\rho\|_{\ga}\lesssim_{N,\ga,s} \Ga(u).
		\ee
		
		\medspace
		
		{\bf Step II:} We now quantitatively control $|\al-1|$. By our assumption, $\Ga(u):=\|\De u+\frac{\ga}{|x|^2}u+\frac{u^{2^*(s)-1}}{|x|^s}\|_{\dot{H}^{-1}}\leq \delta_0$. Thus using Taylor series expansion and \eqref{OR1} we have,
		\bea
		&\,&\prescript{}{\dot{H}^{-1}}{\langle}\De u+\frac{\ga}{|x|^2}u+\frac{u^{2^*(s)-1}}{|x|^s},u{\rangle}_{\dot{H}^1} \leq \Ga(u)\|u\|_{\ga}\leq \de_0\|u\|_{\ga}\no\\
		&&\implies \int_{\Rn}\frac{u^{2^*(s)}}{|x|^s}\,{\rm d}x \leq \de_0\|u\|_{\ga}+\|u\|_{\ga}^2\no\\
		&&\implies \int_{\Rn}\frac{(\sigma+\rho)^{2^*(s)}}{|x|^s}\,{\rm d}x \leq \|\sigma+\rho\|_{\ga}^2+C(\ga,s,N)\de_0^2\no\\
		&&\implies \al^{2^*(s)}\int_{\Rn}\frac{|U_{\ga,s}^{\la}|^{2^*(s)}}{|x|^s}\,{\rm d}x +p\int_{\Rn}\frac{\sigma^{2^*(s)-1}\rho}{|x|^s}\,{\rm d}x + O(\Ga^2)\leq \left[\al^2\|U_{\ga,s}^{\la}\|_{\ga}^2+\|\rho\|_{\ga}^2\right]+C(\ga,s,N)\de_0\no\\
		&&\implies \al^{2^*(s)}\mu_{\ga,s}(\Rn)^{\tfrac{N-s}{2-s}}\leq \al^2\mu_{\ga,s}(\Rn)^{\tfrac{N-s}{2-s}}+O(\Ga^2)+C(N,\ga,s)\de_0.\no
		\eea
		Hence, $|\al-1|\leq O(\Ga^2)$. Thus, if we set $\tilde{\rho}= \rho+(\al-1)U_{\ga,s}^{\la},$ we get our desired result,
		\be\no
		u= U_{\ga,s}^{\la}+\tilde{\rho},
		\ee
		where $\|\tilde{\rho}\|_{\ga}\leq \|\rho\|_{\ga}+|\al-1|\|U_{\ga,s}^{\la}\|_{\ga}\lesssim_{N,\ga,s} \Ga(u)$.
	\end{proof}
	
	\medskip

	\section{Appendix}\label{SAPP}

    \subsection{Nondegeneracy and spectral properties of the linearized operator}
	
	In this subsection, we collect a few important lemmas which will be needed for the proof related to the eigenvalues and eigenfunctions of the linearized operator $$\mathcal{L}_{\ga,s}:=\frac{\left(-\De-\frac{\ga}{|x|^2}\right)}{\tfrac{U_{\ga,s}^{2^*(s)-2}}{|x|^s}}.$$
	We know that if $U_{\ga,s}$ is of the form \eqref{HSBU}, where the constant $C_{N,\ga,s}$ is chosen so that $\|U_{\ga,s}\|_{\ga}=1$. Then for any $\la>0,$
	$$U_{\ga,s}^{\la}(x)= \la^{\tfrac{N-2}{2}}U_{\ga,s}(\la x)$$ 
	also solves \eqref{MainEq}. This fact indicates that the solution $U_{\ga,s}$ must degenerate. In other words, the kernel of the linearized operator contains non-trivial elements. In \cite{R}, the author shows that the degeneracy happens only along a 1-dimensional subspace characterized by
	\be\no
	Z_{\ga,s}:= \frac{\partial}{\partial\la}\bigg{|}_{\la=1}U_{\ga,s}^{\la} = \sum_{i}x^i\partial_{i}U_{\ga,s} + \tfrac{N-2}{2}U_{\ga,s} \in \dot{H}^1(\Rn).
	\ee
	$Z_{\ga,s}$ solves the eigenvalue problem 
	\be\label{LE}
	-\De \phi -\frac{\ga}{|x|^2}\phi = (2^*(s)-1)\frac{U_{\ga,s}^{2^*(s)-2}}{|x|^s}\phi\qquad\text{in }\dot{H}^1(\Rn),
	\ee
	and the degeneracy to the solution space to \eqref{LE} can occur only along the directions $Z_{\ga,s}$.
	
	\begin{theorem}[Theorem $0.1$, \cite{FR}]
		We assume that $\ga \geq 0$ and that $\ga + s>0.$ Then 
		\be\no
		K :=\left\{\phi \in \dot{H}^1(\Rn)\,:\,-\De \phi -\frac{\ga}{|x|^2}\phi = (2^*(s)-1)\frac{U_{\ga,s}^{2^*(s)-2}}{|x|^s}\phi\,\,\text{in }\dot{H}^1(\Rn)\right\} = \mathbb{R}Z_{\ga,s}.
		\ee
		
		In other words, the kernel of $\left(-\De-\tfrac{\ga}{|x|^2}-(2^*(s)-1)\frac{U_{\ga,s}^{2^*(s)-2}}{|x|^s}\right)$ is one dimensional.
	\end{theorem}
	As a result, we get complete information on the first and second eigenvalues and the corresponding eigenspaces of the operator $\mathcal{L}_{\ga,s}$.
	
%	\begin{proposition}\label{Spec}
%		The first and the second eigenvalues of the operator $\mathcal{L}_{\ga,s}$ are respectively $1$ and $2^*(s)-1$. Moreover, the first eigenvalue is simple and spanned by $U_{\ga,s}$ and the second eigenspace is also one dimensional spanned by $Z_{\ga,s}$ defined above.
%	\end{proposition}
	
%	To prove the above proposition, first we set $w=\tfrac{U_{\ga,s}^{2^*(s)-2}}{|x|^s}$ and consider the corresponding weighted $L^2$ space by
%	\be\no
%	L^2_w(\Rn):=\left\{u\,:\,\int_{\Rn}|u|^2w\,{\rm d}x < \infty\right\}.
%	\ee
	We now recall the following results due to Glaudo-Figalli \cite{FG}:
	
	\begin{proposition}[\cite{FG}, Proposition A.1]
		Given a positive integer $n\in\mathbb{N},$ let $1\leq p<N$ and $q<p^*$ be two real numbers. For any positive weight $w\in L^{\left(\tfrac{p^*}{q}\right)'}(\Rn)$, the following compact embedding holds:
		\be\no
		\dot{W}^{1,p}(\Rn)\xhookrightarrow[cpt]{}L^q_w(\Rn).
		\ee
	\end{proposition}
	\begin{theorem}[\cite{FG}, Proposition A.2]
		For any $N\geq 3,$ and any positive weight $w\in L^{\tfrac{N}{2}}(\Rn)$, the inverse operator $\left(\frac{-\De}{w}\right)^{-1}$ is well-defined and continuous from $L^2_w(\Rn)$ into $\dot{H}^1(\Rn)$. Hence it is a compact self-adjoint operator from $L^2_w(\Rn)$ into itself.	
	\end{theorem}
	
	Since $\frac{U_{\ga,s}^{2^*(s)-2}}{|x|^s}\in L^{\tfrac{N}{2}}(\Rn)$, and $0<\ga<\ga_{H}$, we can prove similar results for the operator $\mathcal{L}_{\ga,s}^{-1}$ in the same spirit of the above results.

%	{\bf Proof of Proposition \ref{Spec}}
	
%	From the previous theorem, we conclude that all the eigenvalues of $\mathcal{L}_{\ga,s}^{-1}$ are real, positive and the corresponding eigenfunctions constitute an orthonormal basis $L^2_w(\Rn).$ We have for $\eta\neq 0,\,\mathcal{L}_{\ga,s}^{-1} v= \eta v$ if and only if $\mathcal{L}_{\ga,s}v =\mu v$ where $\mu=\tfrac{1}{\eta}.$ Hence, there exists a non-decreasing sequence of positive eigenvalues $\mu_n$ and corresponding eigenfuntions $\{\Psi_n\}$ constitutes a basis for $\dot{H}^1(\Rn)$. The first eigenvalue of $\mathcal{L}_{\ga,s}$ is simple and the Rayleigh quotient characterizes the eigenvalues.
 
%    \medskip
    
%   \subsection{Spectral properties of the linearized operator}
 
	First we claim that the weight function $\frac{(U_{\ga,s}^{\la})^{2^*(s)-2}}{|x|^s}\in L^{\frac{N}{2}}(\Rn)$, for any $\la>0$. To see this we observe that, from \cite{CC} when $\ga=0,\,s\in (0,2)$ we have,
	$$U_{0,s}(x)= C_{N,s} \left(1+|x|^{2-s}\right)^{-\tfrac{N-2}{2-s}}$$
	$U_{0,s}(x)=\frac{k_0}{(1+|x|^{2-s})^{\tfrac{N-2}{2-s}}}$, where $k_0$ is chosen so that $\|U_{0,s}\|_{\ga}=1.$\\
	
	Thus for any large $R>0$ we get
	\be\no
	\int_{\Rn}\frac{|U_{0,s}(x)|^{\tfrac{N(2-s)}{(N-2)}}}{|x|^{\tfrac{sN}{2}}}\,{\rm d}x =\int_{\Rn}\frac{|U_{0,s}(x)|^{\tfrac{N(2-s)}{(N-2)}}}{|x|^{\tfrac{sN}{2}}}\,{\rm d}x = \int_{B_{R}}\cdots \,+\,\int_{B_{R}^c} \cdots\no
	\ee
	\bea
	\int_{B_R}\frac{|U_{0,s}(x)|^{\tfrac{N(2-s)}{(N-2)}}}{|x|^{\tfrac{sN}{2}}}\,{\rm d}x &\leq& C\int_{B_{R}}\frac{{\rm d}x}{(1+|x|^{2-s})|x|^{\tfrac{sN}{2}}}\no\\
	&\leq& C\int_{B_{R}}\frac{{\rm d}x}{|x|^{\tfrac{sN}{2}}}<\infty,\qquad\text{as } s<2.\no\\
	\text{Similarly,} \int_{B_{R}^c}\frac{|U_{0,s}(x)|^{\tfrac{N(2-s)}{(N-2)}}}{|x|^{\tfrac{sN}{2}}}\,{\rm d}x &\leq& C\int_{B_{R}^c}\frac{{\rm d}x}{|x|^{(2-s)N+\tfrac{sN}{2}}}<\infty.\no
	\eea
	The last integral is finite since, $(2-s)N+\tfrac{sN}{2}=2N-\tfrac{sN}{2}= N+(1-\tfrac{s}{2})N>N.$ Thus the weight function $\frac{U_{0,s}^{2^*(s)-2}}{|x|^s}\in L^{\tfrac{N}{2}}(\Rn)$.
	
	\begin{remark}
		A spherical harmonic $P_n$ of order $n$ is the restriction to $\mathbb{S}^{N-1}$ of a homogeneous harmonic polynomial of degree $n$. Moreover, it satisfies 
		$$-\De_{\mathbb{S}^{N-1}}P_n=\la_n P_n,$$
		for all $n\in\mathbb{N}_0,$ where $\la_n= \left(n^2+(N-2)n\right)$ are the eigenvalues of the Laplace Beltrami operator on $\mathbb{S}^{N-1}$ with corresponding eigenspace dimension $c_n$. We note that $\la_n\geq 0,\,\la_0=0,\,c_0=1,\,c_1=N$ and 
		\be\no
         c_n= {N+n-1\choose n} - {N+n-3\choose n-2}\,\text{ for }n\geq 2.
         \ee
	\end{remark}
	\begin{lemma}\label{STLgar}
		Let $\eta_i,\,i=1,\,2,\,3,\,\cdots,$ denote the eigenvalues of $\mathcal{L}_{\ga,s}$ given in increasing order. Then, $\eta_1=\mu_{\ga,s}(\Rn)^{\tfrac{2^*(s)}{2}}$ is simple with eigenfunction $U_{\ga,s}^{\la}$ and $\eta_2=(2^*(s)-1)\mu_{\ga,s}(\Rn)^{\tfrac{2^*(s)}{2}}$ and the corresponding eigenspace is again one dimensional spanned by $\left\{\frac{{\rm d}}{{\rm d}\la}\bigg{|}_{\la=1}U_{\ga,s}^{\la}\right\}$. Furthermore, the eigenvalues do not depend on $\la$.
	\end{lemma}
	
	\begin{proof}
		Consider the Rayleigh quotient
		$$R_{U_{\ga,s}}(v)=\frac{\|v\|_{\ga}^2}{\int_{\Rn}\frac{|U_{\ga,s}|^{2^*(s)-2}}{|x|^s}|v|^2\,{\rm d}x}.$$
		If we consider, the transformation $U^{\la}(x)=\la^{\tfrac{N-2}{2}}U(\la x),$ for $\la>0$, then 
		$$R_{U_{\ga,s}^{\la}}(v^{\la})=R_{U_{\ga,s}}(v)\,\text{ for every }v\in\dot{H}^1(\Rn).$$
		The scaling argument above shows that the eigenvalues do not depend on $\la$. In particular, we can assume $\la=1$, and hence $U_{\ga,s}^{\la}=U_{\ga,s}$. Therefore, we would like to solve the eigenvalue problem 
		$$-\De v -\frac{\ga}{|x|^2}v = \eta \frac{U_{\ga,s}^{2^*(s)-2}}{|x|^s}v\qquad\text{ in }\dot{H}^1(\Rn).$$
		Putting $v=R(r)Y(\theta)$ and using spherical co-ordinates we get
		$$-\De \equiv -\De_{r}-\frac{1}{r^2}\De_{\mathbb{S}^{N-1}},$$
		where $-\De_{\mathbb{S}^{N-1}}$ denotes the Laplace-Beltrami operator on $\mathbb{S}^{N-1}.$ Therefore, the separation of variables yields,
		\bea
		-\De_r v-\frac{1}{r^2}\De_{\mathbb{S}^{N-1}}v-\frac{\ga}{r^2}v &=&\la \frac{U_{\ga,s}^{2^*(s)-2}}{r^s}v\no\\
		-\left(\frac{\partial^2}{\partial r^2}+\frac{N-1}{r}\frac{\partial}{\partial r}+\frac{\ga}{r^2}\right)R(r)Y(\theta)+ \frac{R(r)}{r^2}\De_{\mathbb{S}^{N-1}}Y(\theta)&=& \la\frac{|U_{\ga,s}|^{2^*(s)-2}}{r^s}R(r)Y(\theta)\no\\
		\frac{R''(r)+\tfrac{N-1}{r}R'(r)+\tfrac{\ga}{r^2}R(r)+\eta\tfrac{U_{\ga,s}^{2^*(s)-2}}{r^s}R(r)}{r^{-2}R(r)}&=&\frac{-\De_{\mathbb{S}^{N-1}}Y(\theta)}{Y(\theta)}=\mu\text{ (say)}.\no
		\eea
		Thus we have,
		\bea
		&-\De_{\mathbb{S}^{N-1}}Y=\mu Y\qquad\text{ on }\mathbb{S}^{N-1}\label{A.1}\\
		&R''+\frac{N-1}{r}R'+\frac{(\ga-\mu)}{r^2}R+\eta \frac{U_{\ga,s}^{2^*(s)-2}}{r^s}R=0,\qquad \text{on }\R_{+}.\label{A.2}
		\eea
		Let us begin with \eqref{A.1}. The first eigenvalue is $\mu_1=0$ and the corresponding eigenfunction is a constant function. The second eigenvalue is $\mu_2=(N-1)$ and the corresponding eigenspace is $N$-dimensional and spanned by $x_i,\,i=1,\,2,\cdots,\,N.$
		
		The second eigenvalue problem is regular at the origin and infinity. If $\mu=\mu_1=0,$ then $\eta_1^1=\mu_{\ga,s}(\Rn)^{\tfrac{2^*(s)}{2}}$ and the corresponding eigenfunction is $U_{\ga,s}$. The second eigenvalue is $\eta_2^1=(2^*(s)-1)\mu_{\ga,s}(\Rn)^{\tfrac{2^*(s)}{2}}$ and the corresponding eigenspace is one dimensional and spanned by $R_2^1=\frac{{\rm d}}{{\rm d}\la}\big{|}_{\la=1}U_{\ga,s}^{\la}.$ The last claim follows from the standard theory for Sturm-Liouville problems with separable boundary conditions. Note that $R_{2}^{1}$ has exactly one zero.
			
			\medspace
			If $\mu=\mu_2=N-1,$ then $\eta_1^2=(2^*(s)-1)\mu_{\ga,s}(\Rn)^{\tfrac{2^*(s)}{2}}$ and $R_1^2=\frac{{\rm d}}{{\rm d}\la}\big{|}_{\la=1}U_{\ga,s}^{\la}$ is the corresponding eigenfunction.
			
			Finally, for $\mu>\mu_2$, monotonicity yields, $\eta_1^k>(2^*(s)-1)\mu_{\ga,s}(\Rn)^{\tfrac{2^*(s)}{2}},$ for all $k>2.$
			
			Since, $\{R_m^k Y_m\}_{m,k}$ span $L^2\left(\frac{U_{\ga,s}^{2^*(s)-2}}{|x|^s},\Rn\right)$ the claim follows from the fact above.
	\end{proof}

	\medskip
	
	\subsection{Interaction estimates}
 
	In this appendix, we like to compute integral quantities involving two bubbles. Fix $N\geq 3,\,0<\ga<\ga_{H},\,s\in (0,2)$. Here, 
	\bea
	U_{\ga,s}(x) &:=& C_{N,\ga,s}\left(|x|^{\tfrac{2-s}{N-2}\ba_{-}(\ga)} + |x|^{\tfrac{2-s}{N-2}\ba_{+}(\ga)}\right)^{-\left(\tfrac{N-2}{2-s}\right)},\no\\
	U_{\ga,s}^{\la}(x) &:=& \la^{\tfrac{N-2}{2}}U_{\ga,s}(\la x),\qquad\text{for }\la>0,\no\\
	\beta_{\pm}(\ga)&:=& \frac{N-2}{2}\pm \var,\no\\
	\var &:=& \sqrt{\left(\frac{N-2}{2}\right)^2-\ga}.\no
	\eea
	First, we consider two bubbles where the first one $\la=1$ and the other has scaling factor $\la\in (0,1]$.
	\begin{lemma}\label{BI}
		Given $N\geq 3,$ let us fix $\theta+\eta=2^*(s),$ with $\theta,\,\eta\geq 0,\,\la\in(0,1],$
		\be\no
		\int_{\Rn}\frac{U_{\ga,s}(x)^{\theta} U_{\ga,s}^{\la}(x)^{\eta}}{|x|^s}\,{\rm d}x\approx_{N,\ga,s}\begin{cases}
			\la^{\var\min\{\theta,\,\eta\}}\qquad\text{if }|\theta-\eta|\geq \de\\
			\la^{\var\tfrac{N-s}{N-2}}\log(\la^{-1})\,\,\text{if }\theta=\eta=\frac{2^*(s)}{2},
		\end{cases}
		\ee
		for some $\delta>0$.
	\end{lemma}
	
	\begin{proof}
		On $B(0,\la^{-1}),\,|\la x|<1,\,\la\in (0,1)$.
		\be\no
		U_{\ga,s}^{\la}(x)\approx_{N,\ga,s} \la^{\tfrac{N-2}{2}-\beta_{-}(\ga)}|x|^{-\beta_{-}(\ga)}\qquad \text{on }B(0,\la^{-1}).
		\ee
		On $B(0,\la^{-1})^c,\,|x|\geq \la^{-1},$ thus $\tfrac{1}{|x|}\leq \la<1$ as $\la\in (0,1]$. Then
		\bea
		U_{\ga,s}(x) &\approx_{N,\ga,s}& \frac{1}{|x|^{\beta_{-}(\ga)}\left(1+|x|^{2\varepsilon \tfrac{2-s}{N-2}}\right)^{\tfrac{N-2}{2-s}}}\no\\
		&\approx_{N,\ga,s}& \frac{1}{|x|^{\beta_{-}(\ga)+2\varepsilon}\left(1+\frac{1}{|x|^{2\varepsilon \tfrac{2-s}{N-2}}}\right)^{\tfrac{N-2}{2-s}}}\approx_{N,\ga,s}|x|^{-\beta_{+}(\ga)}\quad\text{on }B(0,\la^{-1})^c.\no
		\eea
		On the other hand, on $B(0,\la^{-1})^c,$
		\bea
		U_{\ga,s}^{\la}(x) &\approx_{N,\ga,s}& \frac{\la^{\tfrac{N-2}{2}}}{|\la x|^{\beta_{-}(\ga)}\left(1+|\la x|^{2\varepsilon\tfrac{2-s}{N-2}}\right)^{\tfrac{N-2}{2-s}}}\no\\
		&\approx_{N,\ga,s}&\la^{\tfrac{N-2}{2}-\beta_{+}(\ga)}|x|^{-\ba_{+}(\ga)}\no\\
		&\approx_{N,\ga,s}&\la^{-\varepsilon}|x|^{-\beta_{+}(\ga)}.\no
		\eea
		Since we have $\theta+\eta = 2^*(s),$ we get
		\bea
		&\,&\int_{\Rn}\frac{U_{\ga,s}(x)^{\theta} U_{\ga,s}^{\la}(x)^{\eta}}{|x|^s}\,{\rm d}x =\int_{B(0,\la^{-1})}\frac{U_{\ga,s}(x)^{\theta} U_{\ga,s}^{\la}(x)^{\eta}}{|x|^s}\,{\rm d}x+	\int_{B(0,\la^{-1})^c}\frac{U_{\ga,s}(x)^{\theta} U_{\ga,s}^{\la}(x)^{\eta}}{|x|^s}\,{\rm d}x\no\\
		&\approx_{N,\ga,s}& \bigints_{t=0}^{\la^{-1}}\frac{\la^{\varepsilon\eta}t^{N-1-s}}{t^{\beta_{-}(\ga)\tfrac{2(N-s)}{2-s}}\left(1+t^{2\varepsilon\tfrac{2-s}{N-2}}\right)^{\tfrac{N-2}{2-s}\theta}}\,{\rm d}t + \bigints_{t=\la^{-1}}^{\infty}t^{N-1-s-\beta_{+}(\ga)\tfrac{2(N-s)}{(N-2)}}\la^{-\var\eta}\,{\rm d}t\no\\
		&\approx_{N,\ga,s}& \la^{\varepsilon \eta}\int_{t=1}^{\la^{-1}}t^{N-1-s-\beta_{-}(\ga)2^*(s)-2\varepsilon\theta}\,{\rm d}t + \la^{\varepsilon\theta}\no
		\eea
		Thus whenever $\theta=\frac{N-s-\ba_{-}(\ga)2^*(s)}{2\var}=\frac{N-s}{N-2}=\frac{2^*(s)}{2}$,  $\log$ term will appear.
		
		\begin{itemize}
			\item If $\theta\geq \eta+\delta$ (for some $\delta>0$), the above expression becomes comparable to $\la^{\var\eta}$.
			\item If $\eta\geq \theta+\delta$ (for some $\delta>0$), the above expression becomes comparable to $\la^{\var\theta}$.
			\item If $\theta=\eta=\frac{2^*(s)}{2}$, the above expression becomes comparable to $\la^{\var\tfrac{N-s}{N-2}}\log (\la^{-1})$.
		\end{itemize}
	\end{proof}

    \subsection{Symmetries of the problem}
    For $\la>0$, define,
    \be\no
        T_{\la}:C^{\infty}_{c}(\Rn\setminus \{0\}) \to C^{\infty}_{c}(\Rn\setminus \{0\})
        \ee
        by $T_{\la}(\phi)(x):=\la^{-\tfrac{N-2}{2}}\phi(\tfrac{x}{\la})$. The operator $T_{\la}$ satisfies the following properties:
        \bea
        \int_{\Rn}\frac{T_{\la}(U_{\ga,s})^{\al_1}T_{\la}(U_{\ga,s})^{\al_2}}{|x|^s}\,{\rm d}x &=& \int_{\Rn}\frac{(U_{\ga,s})^{\al_1}(U_{\ga,s})^{\al_2}}{|x|^s}\,{\rm d}x,\quad\text{ for }\al_1,\,\al_2>1,\text{with }\al_1+\al_2=2^*(s),\no\\
        \int_{\Rn}\frac{|T_{\la}(U_{\ga,s})|^{2^*(s)}}{|x|^s}\,{\rm d}x &=& \int_{\Rn}\frac{|U_{\ga,s}|^{2^*(s)}}{|x|^s}\,{\rm d}x,\no\\
        \|T_{\la}(U_{\ga,s})\|_{\ga}^2 &=& \|U_{\ga,s}\|_{\ga}^2,\no\\
        U_{\ga,s}^{\la}=T_{\la}(U_{\ga,s}^{1}) &\text{and}& \partial_{\la}U_{\ga,s}^{\la} = \frac{1}{\la}T_{\la}(\partial_{\la} U_{\ga,s}^{1}).\no
        \eea
        The operator $T_{\la}$ plays a significant role in the study of Hardy-Sobolev inequality as they do not change the two quantities $\|\phi\|_{\ga}$ and $\|\phi\|_{L^{2^*(s)}(\Rn,|x|^{-s})}$. These symmetries help us to consider $U_{\ga,s}^{1}$ instead of considering generic bubble $U_{\ga,s}^{\la}$.

        \medspace
        Using the above symmetries it we can easily show the following :

    \begin{proposition}\label{BubInt}
    For $N\geq 3,$ let $U:=U_{\ga,s}^{\la_1}$ and $V:=U_{\ga,s}^{\la_2}$ be two bubbles with $\la_1\geq \la_2$ and define $Q:=Q(\la_1,\la_2)=\frac{\la_2}{\la_1}$. Then for any fixed $\de>0$ we have,
    \be\no
		\int_{\Rn}\frac{U(x)^{\theta} V(x)^{\eta}}{|x|^s}\,{\rm d}x\approx_{N,\ga,s}\begin{cases}
			Q^{\var\min\{\theta,\,\eta\}}\qquad\text{if }|\theta-\eta|\geq \de\\
			Q^{\var\tfrac{N-s}{N-2}}\log(\la^{-1})\,\,\text{if }\theta=\eta=\frac{2^*(s)}{2},
		\end{cases}
		\ee
        
  \end{proposition}

  %  \begin{proof}
  %      The result follows from the following facts. Define,
  %      \be\no
  %      T_{\la}:C^{\infty}_{c}(\Rn\setminus \{0\}) \to C^{\infty}_{c}(\Rn\setminus \{0\})
  %      \ee
  %      by $T_{\la}(\phi)(x):=\la^{-\tfrac{N-2}{2}}\phi(\tfrac{x}{\la})$, and observing the facts that
  %      \bea
  %      \int_{\Rn}\frac{T_{\la}(U_{\ga,s})^{\al_1}T_{\la}(U_{\ga,s})^{\al_2}}{|x|^s}\,{\rm d}x &=& \int_{\Rn}\frac{(U_{\ga,s})^{\al_1}(U_{\ga,s})^{\al_2}}{|x|^s}\,{\rm d}x,\quad\text{ for }\al_1,\,\al_2>1,\text{with }\al_1+\al_2=2^*(s),\no\\
  %      \int_{\Rn}\frac{|T_{\la}(U_{\ga,s})|^{2^*(s)}}{|x|^s}\,{\rm d}x &=& \int_{\Rn}\frac{|U_{\ga,s}|^{2^*(s)}}{|x|^s}\,{\rm d}x,\no\\
  %      \|T_{\la}(U_{\ga,s})\|_{\ga}^2 &=& \|U_{\ga,s}\|_{\ga}^2.\no
  %      \eea
  %      Now using Lemma \ref{BI} the result follows.
  %  \end{proof}
        
	\begin{remark}
		The above results match the results of Figalli-Glaudo \textup{(\cite{FG}, Lemma $B.1$ and Proposition $B.2$)} for the case $\ga=s=0$.
	\end{remark}

    \medskip
    
	{\bf Acknowledgement:} S. Chakraborty is supported by IPD Fellowship from the Department of Mathematics, IIT Bombay. The author would like to thank Prof. Saikat Mazumdar, Department of Mathematics, IIT Bombay, for many discussions and for sharing his thoughtful insights during the writing of this article.


\begin{thebibliography}{XX}
    \bibitem{SA}{\sc Aryan, S.}, Stability of Hardy Littlewood Sobolev Inequality under Bubbling, {\em Calc. Var. Partial Differential Equations} 62 (2023), no.8, Paper No. 223., arXiv:2109.12610 [math.AP].
     
	\bibitem{AT}{\sc Aubin, T.}, Problèmes isopérimétriques de Sobolev. {\em J. Differ. Geom.} 11 (1976), no. 4, 573–598.

    \bibitem{BGKM}{\sc Bhakta, M.; Ganguly, D.; Karmakar, D.; Mazumdar, S.}, {\em Sharp quantitative stability of Poincare-Sobolev inequality in the hyperbolic space and applications to fast diffusion flows}, 
	\href{https://doi.org/10.48550/arXiv.2207.11024}{arxiv}.
	
	\bibitem{BGKM1}{\sc Bhakta, M.; Ganguly, D.; Karmakar, D.; Mazumdar, S.}, {\em Sharp quantitative stability of Struwe's decomposition of the Poincaré-Sobolev inequalities on the hyperbolic space: Part I}, \href{https://doi.org/10.48550/arXiv.2211.14618}{arXiv:2211.14618}.
		
	\bibitem{BS}{\sc Bhakta, M.; Sandeep, K.}, Hardy-Sobolev-Maz'ya type equations in bounded domains, {\em J. Differential Equations}, 247 (2009), no. 1, 119--139.
	
	\bibitem{BE}{\sc Bianchi, G.; Egnell, H.}, A note on the Sobolev inequality. {\em J. Funct. Anal.} 100 (1991), no. 1, 18–24.
	
	\bibitem{BL}{\sc Brezis, Haïm; Lieb, Elliott H.}, Sobolev inequalities with remainder terms. {\em J. Funct. Anal.} 62 (1985), no. 1, 73–86.
	
	\bibitem{CGS}{\sc Caffarelli, L. A.; Gidas, B.; Spruck, J.}, Asymptotic symmetry and local behavior of semilinear elliptic equations with critical Sobolev growth. {\em Comm. Pure Appl. Math.} 42(1989), no.3, 271–297.

    \bibitem{CKN}{\sc Caffarelli, L. A.; Kohn, R.; Nirenberg, L.}, First order interpolation inequalities with weights. {\em Compositio Math.} 53(1984), no.3, 259–275.
	
	\bibitem{CF}{\sc Carlen E.A.; Figalli A.}, Stability for a GNS inequality and the log-HLS inequality, with application to
	the critical mass Keller-Segel equation. {\em Duke Math. J}, 162 (2013), no. 3, 579–625.
	
	\bibitem{CW}{\sc Catrina, F.; Wang, Z.Q.}, On the Caffarelli-Kohn-Nirenberg inequalities: sharp constants, existence (and nonexistence), and symmetry of extremal functions. {\em Comm. Pure Appl. Math.}, 54(2001), no.2, 229–258.

    \bibitem{CC}{\sc Chou, K. S.; Chu, C. W.}, On the best constant for a weighted Sobolev-Hardy inequality.{\em J. London Math. Soc.} (2) 48 (1993), no. 1, 137–151.
 
    \bibitem{CA}{\sc Cianchi A.}, A quantitative Sobolev inequality in BV, {\em J. Funct. Anal}, 237 (2006), no. 2, 466–481.
	
	\bibitem{CFMP}{\sc Cianchi A.; Fusco N.; Maggi F.; Pratelli A.}, The sharp Sobolev inequality in quantitative form. {\em J. Eur. Math. Soc. \textup{(JEMS)}} 11 (2009), no. 5, 1105–1139.

%    \bibitem{DELT}{\sc Dolbeault, J.; Esteban, M. J.; Loss, M.; Tarantello, G.}, On the symmetry of extremals for the Caffarelli-Kohn-Nirenberg inequalities. {\em Adv. Nonlinear Stud.}, 9(2009), no.4, 713–726.
	
	\bibitem{CFM}{\sc Ciraolo G.; Figalli A.; Maggi F.}, A Quantitative Analysis of Metrics on Rn with Almost Constant Positive Scalar Curvature, with Applications to Fast Diffusion Flows, {\em International Mathematics Research Notices}, Volume 2018, Issue 21, November 2018, Pages 6780–6797, https://doi.org/10.1093/imrn/rnx071.

    \bibitem{DSW}{\sc Deng B.; Sun L.; Wei J.}, Sharp quantitative estimates of Struwe's Decomposition. {\em arXiv:2103.15360}.
	
	\bibitem{DWY}{\sc Ding, W. Y.}, On a conformally invariant elliptic equation on Rn. {\em Comm. Math. Phys.} 107(1986), no.2, 331–335.
	
	
	\bibitem{DELT}{\sc Dolbeault, J.; Esteban, M. J.; Loss, M.; Tarantello, G.},
	On the symmetry of extremals for the Caffarelli-Kohn-Nirenberg inequalities. {\em Adv. Nonlinear Stud.} 9(2009), no.4, 713–726.
	
	
	\bibitem{FG}{\sc Figalli, A.; Glaudo, F.}, On the Sharp Stability of Critical Points of the Sobolev Inequality. {\em Arch Rational Mech Anal} 237, 201–258 (2020).
	
	\bibitem{FJ}{\sc Figalli, A.; Jerison, D.}, Quantitative stability for the Brunn-Minkowski inequality. {\em Adv. Math.} 314(2017), 1–47.

    \bibitem{FN}{\sc Figalli, A.; Neumayer, R.},	Gradient stability for the Sobolev inequality: the case $p\geq 2$. {\em J. Eur. Math. Soc. \textup{(JEMS)}} 21 (2019), no.2, 319–354.

    \bibitem{FZ}{\sc Figalli, A.; Zhang, Y. R.-Y.},	Sharp gradient stability for the Sobolev inequality. {\em Duke Math. J.} 171(2022), no.12, 2407–2459.

    \bibitem{RLF}{\sc Frank, R.L.}, The sharp Sobolev inequality and its stability: An introduction. 
    \href{https://arxiv.org/abs/2304.03115}{https://doi.org/10.48550/arXiv.2304.03115}.

    \bibitem{GR}{\sc Ghoussoub, N.; Robert, F.}, Sobolev inequalities for the Hardy-Schrödinger operator: extremals and critical dimensions. {\em Bull. Math. Sci.}, 6(2016), no.1, 89–144.

    \bibitem{GY}{\sc Ghoussoub, N.; Yuan, C.}, Multiple solutions for quasi-linear PDEs involving the critical Sobolev and Hardy exponents.{\em Trans. Amer. Math. Soc.}, 352(2000), no.12, 5703–5743.

    \bibitem{GNN}{\sc Gidas, B.; Ni, Wei Ming; Nirenberg, L.}, Symmetry and related properties via the maximum principle. {\em Comm. Math. Phys.} 68(1979), no.3, 209–243.

     \bibitem{LAM}{\sc Lieb, E. H.}, Sharp constants in the Hardy-Littlewood-Sobolev and related inequalities. {\em Ann. of Math.}, (2)118(1983), no.2, 349–374.

    \bibitem{PLL}{\sc Lions, P. L.}, The concentration-compactness principle in the calculus of variations. The limit case. I. {\em Rev. Mat. Iberoamericana}, 1(1985), no.1, 145–201.

    \bibitem{RSW}{\sc Rădulescu, V.; Smets, D.; Willem, M.}, Hardy-Sobolev inequalities with remainder terms. {\em Topol. Methods Nonlinear Anal.}, 20(2002), no.1, 145–149.

    \bibitem{R}{\sc Rey, O.}, The role of the Green's function in a nonlinear elliptic equation involving the critical Sobolev exponent.{\em J. Funct. Anal.} 89 (1990), no. 1, 1–52.
	
	\bibitem{FR}{\sc Robert, F.}, Nondegeneracy of positive solutions to nonlinear Hardy-Sobolev equations. {\em Adv. Nonlinear Anal.} 6 (2017), no. 2, 237–242.
	
	\bibitem{Sm}{\sc Smets, D.}, {\em Nonlinear Schr\"odinger equations with Hardy potential and critical nonlinearities}, {Trans. Amer. Math. Soc.} 357 (2005), no. 7, 2909--2938.
	
	\bibitem{S}{\sc Struwe, M.}, A global compactness result for elliptic boundary value problems involving limiting nonlinearities. {\em Math Z} 187, 511–517 (1984). https://doi.org/10.1007/BF01174186.
	
	\bibitem{SM}{\sc Struwe, M.}, Variational methods. Applications to nonlinear partial differential equations and Hamiltonian systems. {\em Springer-Verlag, Berlin}, 1990. xiv+244 pp.
	
	\bibitem{TG}{\sc Talenti, G.}, Best constant in Sobolev inequality. {\em Ann. Mat. Pura Appl.} (4)110(1976), 353–372.
	
	\bibitem{WW}{\sc Wei, J.; Wu, Y.}, On the stability of the Caffarelli-Kohn-Nirenberg inequality. (English summary){\em Math. Ann.} 384(2022), no.3-4, 1509–1546.
\end{thebibliography}
\end{document}